\documentclass[12pt]{amsart}

\usepackage{setspace}
\usepackage{verbatim}

\usepackage{amsfonts,fancyhdr,ifthen,bm}
\usepackage{amscd,amssymb} 
\usepackage[letterpaper,left=1.35in,right=1.35in,top=1.35in,bottom=1.35in]{geometry}
 
\usepackage{times}
\usepackage{graphicx}
\usepackage{color}
\usepackage{epsfig}
\usepackage{mathrsfs}
\usepackage{paralist}
\usepackage{comment}
\usepackage{mathtools}
\usepackage{multirow}
\usepackage{tikz-cd}
\usepackage[foot]{amsaddr}

\usepackage{enumitem}

\usepackage{appendix}

\usepackage{tikz,ifthen,calc}
\usetikzlibrary{automata,arrows,positioning,calc}

\newtheorem{thm}{Theorem}[section]

\newtheorem*{thm*}{Theorem}
\newtheorem{prop}[thm]{Proposition}
\newtheorem*{prop*}{Proposition}

\newtheorem*{cor*}{Corollary}

\newtheorem{lem}[thm]{Lemma}
\newtheorem*{lem*}{Lemma}

\newtheorem*{oquest*}{Open Question}

\theoremstyle{remark}
\newtheorem{rmk}[thm]{Remark}

\theoremstyle{remark}
\newtheorem*{rmk*}{Remark}

\theoremstyle{definition}
\newtheorem{defn}[thm]{Definition}

\theoremstyle{definition}
\newtheorem{notat}[thm]{Notation}

\theoremstyle{definition}
\newtheorem{ass}[thm]{Assumption}

\theoremstyle{definition}

\theoremstyle{definition}
\newtheorem*{defn*}{Definition}

\theoremstyle{definition}
\newtheorem{ex}[thm]{Example}

\theoremstyle{definition}

\numberwithin{equation}{section}

\DeclareMathOperator\coker{coker}

 \usepackage{chngcntr}
\counterwithin{figure}{section}

\newcommand{\Z}{\mathbb{Z}}
\newcommand{\QQ}{\mathbb{Q}}

\newcommand{\msW}{\mathscr{W}}

\newcommand{\mcH}{\mathcal{H}}

\DeclareMathOperator{\FFF}{\mathbb{F}}

\newcommand{\Sel}{\textup{Sel}}
\newcommand{\Gal}{\textup{Gal}}

\newcommand{\Hom}{\textup{Hom}}

\newcommand{\isoarrow}{\xrightarrow{\,\,\,\sim\,\,\,}}

\newcommand{\res}{\textup{res}}

\newcommand{\mcalH}{\mathcal{H}}

\newcommand{\CTP}{\textup{CTP}}

\newcommand{\inv}{\textup{inv}}

\DeclareFontFamily{U}{wncy}{}
\DeclareFontShape{U}{wncy}{m}{n}{<->wncyr10}{}
\DeclareSymbolFont{mcy}{U}{wncy}{m}{n}
\DeclareMathSymbol{\Sha}{\mathord}{mcy}{"58}

\newcommand\restr[2]{{
  \left.\kern-\nulldelimiterspace 
  #1 
  \vphantom{\big|} 
  \right|_{#2} 
  }}

\newcommand{\SMod}{\textup{SMod}}

\newcommand{\Fstar}{(\Fsep)^{\times}} 
\newcommand{\Fbar}{F^s} 
\newcommand{\Fsep}{F^s} 
\newcommand{\Fvsep}{F_v^s} 

\newcommand{\resprod}{\sideset{}{'}\prod}
\newcommand{\tresprod}[1]{\sideset{}{'} {\textstyle \prod_{#1}}}
\newcommand{\CTPb}[1]{\textup{CTP}^{\textup{bis}}_{#1}}

\usepackage{scalerel}
\newcommand{\llop}{\scaleto{\bigg(}{33pt}}
\newcommand{\rlop}{\scaleto{\bigg)}{33pt}}

\author{Adam Morgan}
\email{adam.morgan@glasgow.ac.uk}
\author{Alexander Smith}
\email{asmith13@stanford.edu}
 
\begin{document}
\title{The Cassels--Tate pairing for finite Galois modules}
 
\date{\today}

\maketitle

\begin{abstract}
Given a global field $F$ with absolute Galois group $G_F$, we define a category $\SMod_F$ whose objects are finite $G_F$-modules decorated with local conditions. We define this category so that `taking the Selmer group' gives a functor $\Sel$ from $\SMod_F$ to $\text{Ab}$. After defining a duality functor $\vee$ on $\SMod_F$, we show that every short exact sequence $E = \left[0 \to M_1 \to M \to M_2 \to 0\right]$ in $\SMod_F$  
 gives rise to a natural bilinear pairing
\[\CTP_E\colon \Sel\, M_2 \times \Sel\, M_1^{\vee} \to \QQ/\Z\]
whose left and right kernels are the images of $\Sel\, M$ and $\Sel\, M^{\vee}$, respectively. This  generalizes the Cassels--Tate pairing defined on the Shafarevich--Tate group of an abelian variety over $F$ and results in a flexible theory in which pairings associated to different exact sequences can be readily compared to one another. As an application, we give a new proof of Poonen and Stoll's results concerning the failure of the Cassels--Tate pairing to be alternating for  principally polarized abelian varieties and extend this work to the setting of Bloch--Kato Selmer groups.
\end{abstract}

\setcounter{tocdepth}{1}

\tableofcontents

\section{Introduction} 
Given an abelian variety $A$ defined over a global field $F$, one of the most important tools number theorists have for understanding the group of rational points $A(F)$ is the exact sequence
\begin{equation}
\label{eq:Sel_exact}
0 \rightarrow A(F)/mA(F) \to \Sel^m A \to \Sha(A)[m] \to 0,
\end{equation}
which is defined for every positive integer $m$. In this sequence, $\Sel^m A$ denotes the $m$-Selmer group of $A$ over $F$, which is effectively computable, and $\Sha(A)$ denotes the Shafarevich--Tate group of $A$ over $F$, which is more mysterious. In particular, the Shafarevich--Tate conjecture, which predicts that $\Sha(A)$ is always a finite group, is still vastly open. Taking $\Sha(A)_{\text{div}}$ to be the set of divisible elements in the Shafarevich--Tate group, one simple consequence of the Shafarevich--Tate conjecture would be
\[\Sha(A)_{\text{div}} = 0,\]
but this too remains out of reach.

The Cassels--Tate pairing is a canonical nondegenerate bilinear pairing
\begin{equation}
\label{eq:OGCT}
\Sha(A)\big/\Sha(A)_{\text{div}} \,\times\, \,\Sha(A^{\vee})\big/\Sha(A^{\vee})_{\text{div}}\,  \rightarrow \QQ/\Z,
\end{equation}
where $A^{\vee}$ denotes the dual abelian variety to $A$. This pairing was first defined by Cassels for elliptic curves over number fields \cite{Cass62}, where it was observed to be an alternating pairing under the natural identification of $A$ with $A^{\vee}$. The pairing was then extended by Tate to the case of abelian varieties over number fields using his newly developed theory of duality in Galois cohomology \cite{Tate63}. 
In the case where $F$ has characteristic $p > 0$, Tate instead defined a nondegenerate pairing
\begin{equation}
\label{eq:OGCT_FF}
\Sha(A)_{p-\text{div}}\big/\Sha(A)_{\text{div}} \,\times\,\, \Sha(A^{\vee})_{p-\text{div}}\big/\Sha(A^{\vee})_{\text{div}}\,  \rightarrow \QQ/\Z,
\end{equation}
where $\Sha(A)_{p-\text{div}}$ denotes the $p$-divisible elements of $\Sha(A)$. Tate's definition is referred to as the Weil pairing definition of the Cassels--Tate pairing \cite{Poon99}, and was later generalized to the Shafarevich--Tate groups of divisible Galois modules over number fields by Flach \cite{Flach90}. The full pairing \eqref{eq:OGCT} over function fields was defined later by generalizing the homogeneous space approach of Cassels, and many different definitions of the pairing are now known \cite{Poon99}. If $A$ is identified with $A^{\vee}$ by a principal polarization defined over $F$, this pairing is known to be antisymmetric but, by the work of Poonen and Stoll \cite{Poon99}, not necessarily alternating.

The Cassels--Tate pairing can be reinterpreted as a pairing for Selmer groups. Given positive integers $m$ and $n$ indivisible by the characteristic of $F$, we can consider the restriction of the Cassels--Tate pairing to $\Sha(A)[n] \times \Sha(A^{\vee})[m]$. Composing with the surjection of \eqref{eq:Sel_exact}, this gives a canonical pairing
\begin{equation}
\label{eq:OGCT_Sel}
\Sel^{n} A \,\times\, \Sel^{m}A^{\vee} \rightarrow \QQ/\Z.
\end{equation}
We can define homomorphisms
\[\Sel^{mn} \,A \xrightarrow{\,m\,} \Sel^{n}\, A\quad\text{and}\quad\Sel^{mn} A^{\vee} \xrightarrow{\,n\,} \Sel^{m} A^{\vee}\]
functorially from the homomorphisms
\[A[mn] \xrightarrow{\,\, \cdot \,m\,\,} A[n]\quad\text{and}\quad A^{\vee}[mn] \xrightarrow{\,\, \cdot\,n\,\,} A^{\vee}[m].\]
The images of the $mn$-Selmer groups under these maps are killed by \eqref{eq:OGCT_Sel}, and we are left with the perfect pairing
\begin{equation}
\label{eq:OGCT_Sel_2}
\Sel^{n}A \big/ \big(m\Sel^{nm}A\big) \,\times \,\Sel^{m}A^{\vee} \big/ \big(n\Sel^{mn}A^{\vee}\big)\,\to\QQ/\Z.
\end{equation}
Interpreted this way, the Cassels--Tate pairing measures the obstruction to lifting a Selmer element to a higher Selmer group. The original pairing \eqref{eq:OGCT} (or \eqref{eq:OGCT_FF} in the function field case) can be recovered by taking the limit of the pairing \eqref{eq:OGCT_Sel_2} over  tuples $(m, n)$.

Compared to the full pairing \eqref{eq:OGCT}, the pairing \eqref{eq:OGCT_Sel_2} has the advantage of being computable just from the Galois structure of the exact sequence
\begin{equation}
\label{eq:Ator_es}
0 \to A[m] \to A[mn] \to A[n] \to 0
\end{equation}
together with the form of the local conditions on $A[mn]$ (see  Example \ref{ex:classical}).  The starting observation of this paper  is that an analogue of the pairing \eqref{eq:OGCT_Sel_2} can be defined for any short exact sequence of finite Galois modules decorated with local conditions, provided the Galois modules have order indivisible by the characteristic of $F$. This results in a flexible theory in which the pairings associated to different exact sequences can be compared  readily to one another.

A major motivation for our study is the second author's work on the distribution of Selmer groups in twist families \cite{Smi22}. In particular, Theorem \ref{thm:CTP_intro} has a central role in that work, and the generalization of Poonen and Stoll's result to Bloch--Kato Selmer groups expands the scope of the main result of that paper.

Our study of the Cassels--Tate pairing begins with the following categorical framework for studying Selmer groups.
\subsection{Selmer groups and the category $\SMod_F$ }
\label{ssec:SModF}

Take $F$ to be a global field with separable closure $F^s$ and absolute Galois group $G_F$. For each place $v$ of $F$, take $G_v$ to be the decomposition group at $v$ considered as a subgroup of $G_F$.

\begin{defn} \label{def:smod_cat}
We define the additive category $\SMod_F$ of ``Selmerable modules'' as follows:
\begin{itemize}
\item The objects of $\SMod_F$ are tuples $(M, \msW)$, where $M$ is a finite discrete $G_F$-module whose order is indivisible by the characteristic of $F$, and where $\mathscr{W}$ is a compact open subgroup of the restricted product
\begin{equation}
\label{eq:restricted_product_intro}
 \resprod_v H^1(G_v, M)
\end{equation}
taken with respect to the subgroups $H^1_{\text{ur}}(G_v, M)$ of unramified classes.  We refer to $\msW$ as the local conditions subgroup decorating $M$.
\item A morphism $f \colon (M, \msW) \to (M', \msW')$ in $\SMod_F$ is a $G_F$-equivariant homomorphism $f \colon M \to M'$ satisfying
$f(\msW) \subseteq \msW'.$
\end{itemize}
\end{defn}
The Selmer group associated to an object $(M,\mathscr{W})$ in $\SMod_F$ is defined by
\[\Sel(M,\,\msW) \,=\, \ker\llop H^1(G_F, M) \xrightarrow{\quad} \,\, \resprod_{v\text{ of } F} H^1(G_v, M)  \Big/\msW\rlop.\]
The morphisms in $\SMod_F$ have been defined so that the assignment $M\mapsto \textup{Sel} ~M$ defines a functor from $\SMod_F$ to the category of finite abelian groups. 
 
 \begin{rmk}
 A subgroup $\msW$ of the restricted product \eqref{eq:restricted_product_intro} is compact open if and only if it has the form
\[ \mathscr{W}_S \times \prod_{v \not \in S} H^1_{\text{ur}}(G_v, M)\]
for some finite set $S$  of places of $F$ and some  subgroup $\mathscr{W}_S$ of $\prod_{v \in S} H^1(G_v, M)$.  
In almost all applications,  $\msW$ can be written as a restricted product $\tresprod{v} \msW_v$. 
 For the theory we develop, there seems to be no natural reason to restrict to this case. 
 
We remark that the condition that $\mathscr{W}$ be compact open can be relaxed; see Section \ref{sec:general_local_conds} for more details.
\end{rmk}

Before defining our version of the Cassels--Tate pairing we need two additional notions: that of the dual of an object of $\SMod_F$, and that of a short exact sequence in $\SMod_F$. For the first, given   $(M,\mathscr{W})$ in  $\SMod_F$,  we take $M^{\vee}$ to be the $G_F$-module 
\[M^{\vee} =  \Hom(M,\, (F^s)^{\times}).\]
We then take $\mathscr{W}^\perp$ to be the orthogonal complement of $\mathscr{W}$ under the  pairing
\begin{equation} \label{eq:restricted_product_pairing}
 \resprod_v H^1(G_v, M)\, \times\,  \resprod_v H^1(G_v, M^\vee) \to \QQ/\Z
\end{equation}
given on components by local Tate duality; see Definition \ref{defn:loc_pair} for details. This is a local conditions subgroup for $M^\vee$, and the map 
sending $(M, \,\msW)$ to $ \left(M^{\vee}, \,\msW^{\perp}\right)$  extends to a contravariant functor $\vee$ on $\SMod_F$.

Next, given a sequence \begin{equation}
\label{eq:example_exact_intro}
E = \left[0 \to (M_1,\,\msW_1) \xrightarrow{\,\,\iota\,\,} (M,\, \msW) \xrightarrow{\,\,\pi\,\,} (M_2,\, \msW_2) \to 0\right]
\end{equation}
in $\SMod_F$, we call $E$  \emph{exact} if the underlying sequence of $G_F$-modules is exact, and if we have 
\begin{equation*}
\label{eq:local_exact}
\iota^{-1}(\msW) = \msW_1\quad\text{and}\quad \pi(\msW) = \msW_2.
\end{equation*}
Applying the contravariant functor $\vee$ to $E$ gives another sequence
\begin{equation}
\label{eq:dual_seq_input}
E^{\vee} = \left[0 \to \left(M_2^{\vee},\,\msW_2^{\perp}\right) \xrightarrow{\,\,\pi^{\vee}\,\,}\left(M^{\vee},\,\msW^{\perp}\right)  \xrightarrow{\,\,\iota^{\vee}\,\,} \left(M_1^{\vee},\,\msW_1^{\perp}\right) \to 0\right]
\end{equation}
in $\SMod_F$, which is exact if and only if $E$ is exact.

 By  a morphism $f:E \rightarrow E'$ between exact sequences, we mean a commutative diagram  
\begin{equation}
\label{eq:nat_setup}
\begin{tikzcd}
E\,=\,\big[\,0 \arrow{r} & M_{1}\arrow{d}{f_1} \arrow{r}{\iota}& M \arrow{d}{ f}    \arrow{r}{\pi}& M_{2} \arrow{d}{ f_2}    \arrow{r} & 0\,\big] \\
E'\,=\,\big[\,0 \arrow{r} & M_{1}' \arrow{r}{\iota'} & M' \arrow{r}{\pi'} & M_{2}' \arrow{r} & 0\,\big]
\end{tikzcd}
\end{equation}
 in $\SMod_F$. 
 
 \subsection{The pairing and its main properties}
 With the above setup in hand, we can now describe our version of the Cassels--Tate pairing and state its key properties.
 
\begin{thm}\label{thm:CTP_intro}
Take an exact sequence $E$ in $\SMod_F$ as in \eqref{eq:example_exact_intro}.  There is then a natural bilinear pairing  
\begin{equation} \label{cassels_tate_pairing_intro}
\CTP_E\colon \Sel\, M_2\, \times\, \Sel\, M_1^{\vee}\to \QQ/\Z, 
\end{equation}
with left kernel $\pi(\Sel\, M)$ and right kernel  $\iota^{\vee}(\Sel\, M^{\vee})$. 

Moreover:

\begin{itemize}
\item[(i)] for all $\phi$ in $\Sel \,M_2$ and all $\psi$ in $\Sel\,M_1^{\vee}$,  we have the duality identity
\begin{equation*}
\label{eq:intro_sym}
\CTP_E(\phi, \psi) = \CTP_{E^{\vee}}(\psi, \beta(\phi)),
\end{equation*}
where $\beta$ is the evaluation isomorphism from $(M_2^{\vee})^{\vee}$ to $M_2$.
\item[(ii)] given a morphism of exact sequences   $f:E\rightarrow E'$ in $\SMod_F$ as depicted in \eqref{eq:nat_setup},  we have the naturality property
\begin{equation*}
\label{eq:natural}
\textup{CTP}_{E } (\phi, \,{f_1^{\vee}}(\psi) ) \,=\, \textup{CTP}_{E'}(f_{2} (\phi), \psi),
\end{equation*}
for all $\phi$ in $\Sel\, M_{2 }$ and $\psi$ in $\Sel\, (M_{1}')^{\vee}$.
\end{itemize}
\end{thm}
We will prove this theorem in Section \ref{sec:def_and_sym}.

\begin{ex}
\label{ex:classical}
Take $A$ to be an abelian variety defined over $F$. Given a positive integer $m$ indivisible by the characteristic of $F$ and a place $v$ of $F$, take
\[\msW_{m, v} = \ker\Big(H^1(G_v, A[m]) \to H^1(G_v, A)\Big).\]
Then the product $\msW_{m} = \prod_{v} \msW_{m, v}$ is a compact open subgroup of $\tresprod{v} H^1(G_v, A[m])$ by \cite[Proposition 4.13]{PR12}.

The $m$-Selmer group of $A$ over $F$ is simply $\Sel\left(A[m], \,\msW_{m}\right)$. Furthermore, given a second positive integer $n$   indivisible by the characteristic of $F$, the $n$-Selmer group of the dual abelian variety $A^{\vee}$ is equal to
\[\Sel (A[n]^{\vee},\,\msW_n^{\perp}),\]
and the pairing associated to the exact sequence
\[0 \to (A[n], \,\msW_n) \to (A[mn],\, \msW_{mn}) \to (A[m],\, \msW_m) \to 0\]
recovers the pairing \eqref{eq:OGCT_Sel_2}. In this way, the construction considered in Theorem \ref{thm:CTP_intro} generalizes the pairing considered by Cassels and Tate.

See Section \ref{ssec:isogeny} for a generalization of this example to isogeny Selmer groups.
\end{ex}

Our definition of the pairing $\CTP_E$ in Theorem \ref{thm:CTP_intro} is a straightforward generalization of the Weil pairing definition of the Cassels--Tate pairing for abelian varieties given in \cite[I.6.9]{Milne86}. Calculating the left kernel of $\CTP_E$ also requires little adjustment from \cite{Milne86}.
 The duality identity \eqref{eq:intro_sym} is the more interesting portion of this theorem. It may be thought of as a generalization of the fact that the classical Cassels--Tate pairing is antisymmetric  for principally polarized abelian varieties,
 with the prefix `anti-' explained by the fact  that the Weil pairing is alternating. While this antisymmetry result has been extended to more general Galois modules with alternating structure \cite{Flach90}, the duality identity \eqref{eq:intro_sym} seems to have gone unobserved for other kinds of modules. 

The naturality property is also fundamental to our theory. The property is a simple generalization of an analogous result for the Cassels--Tate pairing of isogenous abelian varieties \cite[Remark 6.10]{Milne86}, but the broader scope provided by $\SMod_F$ makes it a more powerful statement in the context of our work. In particular, applications of naturality can make use of the fact that $\SMod_F$ is quasi-abelian, a fact we prove in Section \ref{ssec:qab}. This means that an exact sequence $E$ as in \eqref{eq:example_exact_intro} may be pulled back along a morphism $N_2 \to M_2$ or pushed out along a morphism $M_1 \to N_1$ to give another exact sequence. As a consequence, the usual method of defining the Baer sum of extension classes works; given  extensions $E_a$ and $E_b$ of $(M_2,\mathscr{W}_2)$ by $(M_1,\mathscr{W}_1)$ in  $\SMod_F$, there is a natural choice $E_a + E_b$ for the sum of these extensions. Per Proposition \ref{prop:Baer}, naturality then implies the trilinearity relation
\begin{equation}
\label{eq:intro_baer}
\CTP_{E_a + E_b}\, =\, \CTP_{E_a}\,+\, \CTP_{E_b}.
\end{equation}
Since the sequence  $E_a+E_b$ can be simpler than both $E_a$ and $E_b$, this provides a useful method for computing the sum of the two pairings.  Special cases of this identity,  proved by direct cocycle calculations, are recognizable in \cite{MR597871, MR3951582} where they are exploited to study the variation under quadratic twist of the Cassels--Tate pairing on the $2$-Selmer group of an abelian variety. A variant of this construction also appears in \cite[Section 8]{Smi22}.

Additional applications of the general theory, in particular to the study of class groups, are given in \cite{MoSm21c}. In that work one can also find a study of how the pairing \eqref{cassels_tate_pairing_intro} behaves when the global field $F$ is changed.

\subsection{The Poonen--Stoll class for modules with alternating structure}
\label{ssec:theta_intro}
In \cite{Flach90}, Flach generalized the Cassels--Tate pairing to the Shafarevich--Tate groups of certain divisible Galois modules, such as the ones appearing in the study of Bloch--Kato Selmer groups \cite{BlKa07}. For modules with alternating structure, he found that the Cassels--Tate pairing was antisymmetric (i.e. skew-symmetric). In particular, his result implies that the Cassels--Tate pairing
\[\CTP \colon \Sha(A) \times \Sha(A) \to \QQ/\Z\]
associated to a principally polarized abelian variety $A$ over a number field is antisymmetric.

Poonen and Stoll were the first to discover that the pairing corresponding to a principally polarized abelian variety of dimension at least $2$ is not necessarily alternating  \cite{Poon99}. As part of this work, they explicitly constructed a class  $\psi_{\text{PS}}$  in $(\Sha(A)/\Sha(A)_{\text{div}})[2]$ satisfying
\[\CTP(\phi, \phi) = \CTP(\phi, \psi_{\text{PS}}) \quad\text{for all }\, \phi \in \Sha(A),\]
which we refer to as the  \emph{Poonen--Stoll class}.
For divisible modules with alternating structure, Poonen and Stoll observed that the existence of an analogue to $\psi_{\text{PS}}$ is a formal consequence of Flach's result,  and stated that it would be desirable to give a similarly explicit construction of this class \cite[Section 5]{Poon99}. Since the work of Poonen--Stoll is of a geometric nature, it is not obvious how to adapt their arguments to Flach's setting.  We overcome this difficulty using the machinery developed in this paper, giving a new and almost entirely algebraic proof of Poonen--Stoll's result in the process.


To describe the result, we begin by recalling what we need of Flach's work \cite{Flach90}.
\begin{notat}
\label{notat:TV}
Choose a global field $F$ of characteristic not equal to $2$. Choose a free finite rank $\mathbb{Z}_2$-module $T$ equipped with a continuous action of $G_F$ that is  unramified outside a finite set of places. Write $V=T\otimes_{\mathbb{Z}_2}\mathbb{Q}_2$, and suppose we are given a non-degenerate continuous  $G_F$-equivariant alternating pairing 
\begin{equation} \label{eq:pairing_on_V_flach}
\lambda: V\times V \longrightarrow (F^s)^\times,
\end{equation}
under which $T$ is identified with its orthogonal complement. Suppose also that for each place $v$ of $F$ we are given a $\mathbb{Q}_2$-vector space 
\[\mathscr{W}_v\subseteq H^1(G_v,V),\]
which is its own orthogonal complement under the local duality pairing associated to \eqref{eq:pairing_on_V_flach}, and which is equal to  the subgroup of unramified classes in $H^1(G_v,V)$ for all but finitely many places $v$. Take $\msW = \prod_v \msW_v$.
\end{notat}

From this data, Flach defines a Shafarevich--Tate group $\Sha(T,\mathscr{W})$ as a certain subquotient of $H^1(G_F,V/T)$, and defines a Cassels--Tate pairing 
\begin{equation} \label{eq:CT_flach}
\textup{CTP}_\lambda: \Sha(T,\mathscr{W})\times \Sha(T,\mathscr{W})\longrightarrow \mathbb{Q} /\mathbb{Z}.
\end{equation}
In Section \ref{ssec:infinite}, we generalize Flach's definition of the Cassels--Tate pairing, recovering Flach's result that \eqref{eq:CT_flach} is non-degenerate and antisymmetric.

We now give our explicit description of the associated Poonen--Stoll class. 
\begin{defn}
\label{defn:PS_intro}
The pairing \eqref{eq:pairing_on_V_flach} induces a non-degenerate $G_F$-equivariant alternating pairing 
\[P_\lambda:T/2T \times T/2T \longrightarrow \mu_2,\] 
defined by
\[P_\lambda(x,y)=\lambda(x\otimes \tfrac{1}{2}, \,y).\]
Take $U$ to be the $G_F$-set $ (\Fsep)^{\times} \times T/2T$, and define a group operation on $U$ by
\[(\alpha, x) \cdot (\alpha', x') \,=\,  \big(\alpha \alpha' P_\lambda(x, x'),\, x +x'\big).\]
This makes $U$ into a $G_F$-module. We have an exact sequence of $G_F$-modules
\begin{equation}
\label{eq:PS_exact_Flach_seq}
0 \to (F^s)^{\times} \to U \to T/2T \to 0,
\end{equation}
where the maps to and from $U$ are the obvious inclusion and projection. The $G_F$-module extension class of this exact sequence corresponds to an element of $H^1(G_F, (T/2T)^{\vee})$, which we map into $H^1(G_F,V/T)$ via the composition
\[H^1(G_F,(T/2T)^\vee) \cong H^1(G_F,T/2T) \longrightarrow H^1(G_F,V/T).\] 
Here the first map is induced by $P_\lambda$, and the second induced by the map sending $x$ in $T/2T$  to $x\otimes \tfrac{1}{2} \textup{  mod }T $. Denote by $c_\lambda$ the class in $H^1(G_F,V/T)$ constructed this way. 
\end{defn}

\begin{thm} \label{thm:poonen-stoll_into}
The element $c_\lambda$ maps into $\Sha(T,\mathscr{W})[2]$, and we have 
\[\CTP_\lambda(\phi, \phi) = \CTP_\lambda(\phi, c_\lambda)\quad\text{for all }\, \phi \in \Sha(T, \msW).\]
In particular,  $\CTP_\lambda$ is alternating if and only if $c_\lambda$ maps to $0$ in $ \Sha(T,\mathscr{W})$.
\end{thm}

We will prove this result as Theorem \ref{thm:poonen-stoll_precise}. When applied to the $2$-adic Tate module of a principally polarized abelian variety,  it follows from  \cite[Remark 3.3]{PR11} that the explicit description of the obstruction class above agrees with that of Poonen and Stoll.

\begin{rmk}
If $G_F$ acts on $T/2T$ through the orthogonal group associated to a quadratic form  refining  $P_\lambda$,  then Theorem \ref{thm:poonen-stoll_into} implies that \eqref{eq:CT_flach} is alternating. 
Indeed, given such a quadratic form $q$, the map  $T/2T\rightarrow U$ sending $x$ to $(q(x),x)$ gives a $G_F$-equivariant splitting of the sequence \eqref{eq:PS_exact_Flach_seq}. This can be used to recover Cassels' result that the pairing associated to an elliptic curve over a number field is alternating \cite{Cass62}.
\end{rmk}

To prove Theorem \ref{thm:poonen-stoll_into}, we first prove a general version characterizing the failure of the pairing attached to a suitable exact sequence to be alternating, in the presence of an auxiliary object which we refer to as a \textit{theta group}. We present this result as Theorem \ref{thm:theta_main}. Here, the term ``theta group'' is borrowed from the study of abelian varieties. Given a principally polarized abelian variety $A$ over $F$, and given a positive integer $n$ indivisible by the characteristic of $F$, there is a natural choice of central extension
\[1 \to (F^s)^{\times} \to \mcH \to A[n] \to 0\]
for which the associated commutator pairing on $A[n]$ is the Weil pairing. The group $\mcH$ is known as a theta group, and is defined via isomorphisms of certain line bundles on $A$; see \cite{Mumf66} or Section \ref{sec:theta_line_bundle} for further details. We can then prove Theorem \ref{thm:poonen-stoll_into} by constructing a suitable theta group for the $\QQ_2$-vector space $V$.

In Section \ref{ssec:div_theta}, we give an algebraic method for constructing theta groups associated to finite modules with alternating structure. In Section \ref{sec:theta_line_bundle}, we show that this algebraic construction agrees with the usual construction of theta groups associated to an abelian variety. We remark  that the general construction allows structural results concerning the Galois cohomology of abelian varieties proved using theta groups, such as those established in \cite{PR12}, to be transported to the Bloch--Kato setting.  For additional applications of  Theorem \ref{thm:theta_main}, see  \cite[Sections 5 and 6]{MoSm21c}.

\subsection{Acknowledgments}  
We would like to thank L. Alexander Betts, Peter Koymans, Wanlin Li, Melanie Matchett Wood, Carlo Pagano, Bjorn Poonen, Karl Rubin, and Richard Taylor for helpful discussion about this work.  

This work was partially conducted during the period the second author served as a Clay Research Fellow. Previously, the author was supported in part by the National Science Foundation under Award No. 2002011.

The first author is supported by the Engineering and Physical
Sciences Research Council  grant EP/V006541/1. During part of the period in which this work was carried out, they were supported by the Max-Planck-Institut für Mathematik in Bonn and the Leibniz fellow programme at the Mathematisches Forschungsinstitut Oberwolfach, and would like to express their thanks for the excellent conditions provided.

\section{Notation}
\label{sec:notat}
Throughout this paper, $F$ will denote a global field; that is, it will either be a number field or a finite extension of $\FFF_p(t)$ for some rational prime $p$. Fixing a separable closure $\Fsep$ of $F$, we take $G_F$ to be the absolute Galois group $\Gal(\Fsep/F)$.

For each place $v$ of $F$, $F_v$ will denote the completion of $F$ at $v$. Taking $F_v^s$ to be some separable closure of $F_v$, we write $G_v$ for the absolute Galois group $\Gal(F_v^s/F_v)$ and $I_v$ for the inertia subgroup of $G_v$. In the case that $v$ is archimedean, we have $I_v = G_v$. 

For each $v$ we fix an embedding $\Fsep \hookrightarrow \Fvsep$. These embeddings realize the $G_v$ as closed subgroups of $G_F$, allowing us to view $G_{F}$-modules as $G_v$-modules.

Given a positive integer $n$ and an abelian group $B$, we will use the notation $B[n]$ to denote the kernel of the multiplication by $n$ map on $B$ for any integer $n$. If $n$ is indivisible by the characteristic of $F$, we define $\mu_n$ to be the set of $n^{\text{th}}$ roots of unity inside $F^s$.

\begin{notat}[Continuous cochain cohomology]
\label{notat:contcoch}
Given a topological group $G$, a $G$-module will be a discrete abelian group with a continuous action of $G$. Given a $G$-module $M$, $H^i(G, M)$ will denote the continuous group cohomology of $G$ acting on $M$. Our main reference for these groups is \cite{Neuk08}.

We will define these cohomology group using inhomogeneous cochains. Specifically, given $i \ge 0$, define $C^i(G, M)$ to be the set of continuous maps from $G^i$ to $M$. We have a coboundary map
 \[d\colon C^{i}(G, M) \rightarrow C^{i+1}(G, M)\]
defined by
\begin{align*}
df(\sigma_1, \dots, \sigma_{i+1}) = \,&\sigma_1f(\sigma_2, \sigma_3, \dots, \sigma_{i+1}) \\+&\sum_{j = 1}^{i } (-1)^j f(\sigma_1, \dots,\sigma_{j-1},\, \sigma_j\sigma_{j+1}, \sigma_{j+2}, \dots, \sigma_{i+1})\\
+ &(-1)^{i+1} f(\sigma_1, \dots,   \sigma_{i}).
\end{align*}
These maps define a complex
\[0 \to C^0(G, M) \xrightarrow{\,d\,} C^1(G, M)\xrightarrow{\,d\,} C^2 (G, M)\xrightarrow{\,d\,} \dots.\]
The groups $H^i(G, M)$ are then defined to be the cohomology groups of this complex. More specifically, take $Z^i(G, M)$ to be the kernel of the coboundary map $d$ on $C^i(G, M)$, and take $B^{i+1}(G, M)$ to be its image. Taking $B^0(G, M)$ to be $0$, we then have
\[H^i(G, M) = Z^i(G, M)/B^i(G, M)\]
for all $i \ge 0$. We give these groups the discrete topology.

Elements in $C^i(G, M)$ are called inhomogeneous $i$-cochains. The cochains in $Z^i(G, M)$ are called $i$-cocycles, and the elements in $B^i(G, M)$ are called $i$-coboundaries.

Given a closed subgroup $H$ of $G$, we can consider the restriction map
\[\res = \res_H^G\colon C^i(G, M) \to C^i(H, M).\]
This sends cocycles to cocycles and coboundaries to coboundaries, and so defines a restriction map
\[\res \colon H^i(G, M) \to H^i(H, M).\]

Given a place $v$ of $F$ as above and a $G_v$ module $M$, we define the group of unramified $1$-cocycle classes in $M$ by
\begin{equation} \label{eq:unram_conds_def}
H^1_{\text{ur}}(G_v, M) = \ker\left(H^1(G_v, M) \xrightarrow{\,\res\,}H^1(I_v, M)\right).
\end{equation}
Given $\phi$ in $C^i(G_F, M)$ or $H^i(G_F, M)$, we will sometimes write $\phi_v$ as shorthand for the restriction $\res^{G_F}_{G_v}(\phi)$. More often, given $\phi_v \in H^i(G_v, M)$ for each place $v$ of $F$, we will write $(\phi_v)_v$ for the tuple in $\prod_v H^i(G_v, M)$ whose value at $v$ is $\phi_v$.
\end{notat}

\begin{notat}[Cup product on cochains]
\label{notat:cup}
Suppose we have discrete $G$-modules $M_1, M_2, N$ and a bilinear $G$-equivariant pairing 
\[P : M_1 \otimes M_2 \rightarrow N.\]
Given $f_1 \in C^i(G, M_1)$ and $f_2 \in C^j(G, M_2)$, we define the cup product $f_1 \cup_{P} f_2$ in $C^{i+j}(G, N)$ by
\[(f_1 \cup_{P} f_2)(\sigma_1, \dots, \sigma_{i+j}) = P\left(f_1(\sigma_1, \dots, \sigma_i), \,\sigma_1 \sigma_2 \dots \sigma_i f_2(\sigma_{i+1}, \dots, \sigma_{i+j})\right).\]
The standard cohomological cup product is recovered by restricting this construction to cocycles \cite[Proposition 1.4.8]{Neuk08}. We record \cite[Proposition 1.4.1]{Neuk08}, which states that this construction satisfies
\begin{equation}
\label{eq:cobcup}
d\left( f_1 \cup_{P} f_2\right) = df_1 \cup_{P} f_2  \,\,+\,\, (-1)^i \cdot f_1 \cup_P df_2,
\end{equation}
as can be proven directly. If the choice of pairing $P$ is clear, we will sometimes omit it.

Given a $G_F$-module $M$, we will write the cup product constructed from the standard pairing from $M \times M^{\vee}$ to $(\Fsep)^{\times}$ by $\cup_M$.
\end{notat}

\section{The pairing and its main symmetry}
\label{sec:def_and_sym}
 
In this section we define the generalization of the  Cassels--Tate  pairing  of interest to this paper and  prove Theorem \ref{thm:CTP_intro}. As a starting point, we complete our specification of the category $\SMod_F$ introduced in Section \ref{ssec:SModF} by giving the pairing used to define the duality functor $\vee$.

\begin{defn}
\label{defn:loc_pair}
For every place $v$ of $F$, local Tate duality \cite[Theorem 2.1]{Tate63} gives a perfect pairing
\begin{equation} \label{eq:local_tate_pairing}
H^1(G_v, M) \times H^1(G_v, M^{\vee}) \rightarrow \QQ/\Z,
\end{equation}
defined as the composition of the standard cup product map
 \[\cup_{M}: H^1(G_v, M) \times H^1(G_v, M^{\vee}) \rightarrow  H^2\left(G_v, \,(F_v^s)^{\times} \right)\]
 and the local invariant map 
\begin{equation} \label{eq:invariant_map}
\textup{inv}_v: H^2\left(G_v, \,(F_v^s)^{\times} \right) \to \mathbb{Q}/\mathbb{Z}.
\end{equation}
Under this pairing, $H^1_{\text{ur}}(G_v, M)$ and $H^1_{\text{ur}}(G_v, M^{\vee})$ are orthogonal complements at all but finitely many places $v$ \cite[Theorem 7.2.15]{Neuk08}. As a result, summing the pairings \eqref{eq:local_tate_pairing} gives a  continuous bilinear pairing
\begin{equation}
\label{eq:loc_pairing_2}
\resprod_{v \text{ of }F} H^1(G_v, M)\, \times\, \resprod_{v \text{ of }F} H^1(G_v, M^{\vee}) \to\QQ/\Z,
\end{equation}
which  is topologically perfect in the sense that the induced map from $\tresprod{v} H^1(G_v, M^{\vee})$ to the Pontryagin dual of $\tresprod{v} H^1(G_v, M)$ is a topological isomorphism.
\end{defn}

We now give the definition of the pairing \eqref{cassels_tate_pairing_intro}.  This  is essentially identical to the standard Weil pairing definition of the Cassels--Tate pairing for abelian varieties appearing in \cite{Poon99}, \cite{Tate63}, or  \cite{Milne86}. Before parsing this definition, it may be helpful to consider the heuristic explanation of why the Cassels--Tate pairing ought to exist, which we give in Section \ref{ssec:left_kernel}.
\begin{defn}
\label{defn:CTP}
Take
\begin{equation}
\label{eq:first_E}
E\,=\,\left[0 \rightarrow (M_1, \msW_1) \xrightarrow{\,\,\iota\,\,} (M, \msW) \xrightarrow{\,\,\pi\,\,} (M_2, \msW_2) \rightarrow 0\right]
\end{equation}
to be an exact sequence in $\SMod_F$; recall that this means that this sequence is exact as a sequence of Galois modules, that $\iota^{-1}(\msW) = \msW_1$, and that $\pi(\msW) = \msW_2$. We define a pairing
\[\CTP_E \colon \Sel\, M_2 \, \times\, \Sel\, M_1^{\vee}\to \QQ/\Z\]
as follows.

Suppose we are given $\phi$ in $\Sel\, M_2$ and $\psi$ in $\Sel\, M_1^{\vee}$. Choose cocycles $\overline{\phi} \in Z^1(G_{F}, M_2)$ and $\overline{\psi} \in Z^1(G_{F}, M_1^{\vee})$ representing this pair of Selmer elements, and choose $f\in C^1( G_{F}, M)$ satisfying
\[\overline{\phi} = \pi \circ f.\]
Then $\iota^{-1}(df)$ defines a cocycle in $Z^2(G_{F}, M_1)$, and the cup product 
\[\iota^{-1}(df) \,\cup_{M_1}\, \overline{\psi}\]
lies in $Z^3(G_F, \,(F^s)^{\times})$.

Per \cite[(8.3.11)]{Neuk08} or \cite[Corollary 4.18]{Milne86}, we have
\begin{equation}
\label{eq:H3_is_trivial}
 H^3\left(G_F, \,(F^s)^{\times}\right)[n] = 0
\end{equation}
for any positive integer $n$ indivisible by the characteristic of $F$. We may thus choose a cochain $\epsilon$ in $C^2(G_F, (F^s)^{\times})$ so
\[d\epsilon \,=\, \iota^{-1}(df) \,\cup_{M_1}\, \overline{\psi}.\]

Finally, we can choose cocycles $\overline{\phi}_{v, M}$ in $Z^1(G_v, M)$ for each  place $v$ of $F$, so that  $\pi(\overline{\phi}_{v, M}) = \overline{\phi}_v$,
and so that $(\overline{\phi}_{v, M})_v$ represents a class in the local conditions group $\msW$.

 Having made these choices, we find that 
\[\gamma_v \,=\, \iota^{-1}(f_v - \overline{\phi}_{v, M}) \,\cup_{M_1} \,\overline{\psi}_v \,-\, \epsilon_v \]
lies in $Z^2\left(G_v, (F^s)^{\times}\right)$ for each $v$, and we define the pairing by
\[\text{CTP}_E(\phi, \psi) = \sum_{v \text{ of }F} \inv_v(\gamma_v).\]
\end{defn}

We will content ourselves with a sketch of the proof for this next basic proposition.
\begin{prop}
\label{prop:well_def}
The pairing given in Definition \ref{defn:CTP} is well-defined and bilinear. 
 It also satisfies the naturality property of Theorem \ref{thm:CTP_intro} (ii).
\end{prop}

\begin{proof}
We need to show that the value of this pairing $\CTP_E(\phi, \psi)$ does not depend on the choice of tuple
\[(\overline{\psi},\, \overline{\phi},\, f, \,(\overline{\phi}_{v, M})_v, \,\epsilon).\]
For $\Delta$ in $H^2\left(G_F, \,(\Fsep)^{\times}\right)$, reciprocity for the Brauer group gives $\sum_v \inv_v(\Delta_v) = 0$. From this, we see that the pairing does not depend on $\epsilon$.

Since $\msW_1 = \iota^{-1}(\msW)$ is orthogonal to $(\psi_v)_v$ under the local pairing \eqref{eq:loc_pairing_2}, the value also does not depend on the choice of $(\overline{\phi}_{v, M})_v$ projecting to $\msW$.

We can then account for a change in $f$ or $\overline{\psi}$ by adjusting $\epsilon$ and using \eqref{eq:cobcup} , and we can account for a change in $\overline{\phi}$ by adjusting $(\overline{\phi}_{v, M})_v$. So the pairing is well-defined, and bilinearity follows immediately.

For naturality, consider the commutative diagram \eqref{eq:nat_setup}. Given a place $v$ of $F$, and given $g  \in C^1(G_v, M_{1})$ and $g' \in C^1(G_v, (M_{1}')^{\vee})$, the identity
\[g  \cup_{M_{1 }} f_1^{\vee}(g') \,=\, f_1(g ) \cup_{M_{1 }'} g'.\]
follows from the definition of $f_1^{\vee}$. Naturality then follows  from the definition of the Cassels--Tate pairing.
\end{proof}

\begin{rmk}
Group cohomology groups may be defined as the right derived functors of $H^0$, and it is natural to want a definition of the Cassels--Tate pairing that starts from an arbitrary injective resolution of the involved $G_F$-modules, rather than one relying on cochains. Nekov{\'a}{\v{r}} gives an approach to defining the pairing in such a manner for certain infinite modules in \cite[10.1.4]{Neko06}. An advantage of this approach is that it can be extended to handle finite abelian group schemes over $F$ of order divisible by the characteristic of $F$, where Selmer groups must be defined using fppf cohomology. We will return to this alternative setup in \cite{MoSm21b}.
\end{rmk}

\subsection{The left kernel of the Cassels--Tate pairing}
\label{ssec:left_kernel}

We consider the following result:
\begin{prop}
\label{prop:left_kernel}
The left kernel of the pairing given in Definition \ref{defn:CTP} is $\pi(\Sel\, M)$.
\end{prop}
The proof of this proposition requires little adjustment from the classical case, as presented in e.g. \cite[Section I.6]{Milne86}. But it is worth reflecting on why a  bilinear pairing between $\Sel \, M_2$ and $\Sel \, M_1^{\vee}$ exists with this property. The key tool we need for this is the nine-term Poitou--Tate exact sequence as given in \cite[Theorem 3.1]{Tate63} or \cite[Theorem 4.10]{Milne86}. Specifically, for $(M_1, \msW_1)$ in $\SMod_F$, we will use the following two facts:
\begin{enumerate}
\item There is a perfect bilinear pairing of finite abelian groups
 \begin{equation}
\label{eq:first_step_duality}
\textup{PTP}:\Sha^2(M_1) \times \Sha^1(M_1^{\vee}) \to \QQ/\Z,
\end{equation}
where we have taken the notation
\begin{align*}
 &\Sha^2(M_1) \,=\, \ker\llop H^2(G_F, M_1) \rightarrow \resprod_v H^2(G_v, M_1)\rlop \quad\text{and}\\
&\Sha^1(M_1^{\vee})\,=\, \ker\llop H^1(G_F, M_1^{\vee}) \rightarrow \resprod_v H^1(G_v, M_1^{\vee})\rlop.
\end{align*}
\item Take $\msW_{\text{global}}(M_1)$ to be the image of $H^1(G_F, M_1)$ in $\tresprod{v} H^1(G_v, M_1)$, and define $\msW_{\text{global}}(M_1^{\vee})$ similarly. Then, with respect to the local pairing \eqref{eq:loc_pairing_2}, we have
\[ \msW_{\text{global}}(M_1)^{\perp} \,=\, \msW_{\text{global}}(M_1^{\vee}).\]
This duality result implies
\[\left(\msW_{\text{global}}(M_1) + \msW_1\right)^{\perp}\, =\, \msW_{\text{global}}(M_1^{\vee}) \cap \msW_1^{\perp} ,\]
and \eqref{eq:loc_pairing_2} restricts to a continuous nondegenerate pairing
\begin{equation}
\label{eq:second_step_duality}
\frac{ \tresprod{v}H^1(G_v, M_1)}{\msW_{\text{global}}(M_1) + \msW_1} \,\,\times\,\, \frac{\Sel\, M_1^{\vee}}{\Sha^1(M_1^{\vee})} \xrightarrow{\quad} \QQ/\Z.
\end{equation}
\end{enumerate}

Now take $E$ as in \eqref{eq:first_E}. Given an element $\phi$ of $\Sel\big(M_2,\, \msW_2\big)$, we are interested in measuring the obstruction to lifting $\phi$ under the map
\begin{equation}
\label{eq:doesitlift}
\pi: \Sel( M, \, \mathscr{W}) \to \Sel(M_2,\, \msW_2).
\end{equation}
There seem to be two distinct obstructions to this lift. First,  the long exact sequence associated to $E$ gives  an exact sequence
\begin{equation}
\label{eq:CT_obstr1}
H^1(G_F, M) \xrightarrow{\,\,\pi\,\,} H^1(G_F, M_2) \xrightarrow{\,\,\delta\,\,}  H^2(G_F, M_1).
\end{equation}
Since $\phi$ locally maps to an element of $\pi(\msW)$, we find that the image $\delta(\phi)$ lies in $\Sha^2(M_1)$. In particular, $\phi$ can only lift to $H^1(G_{F}, M)$ if its image in $\Sha^2(M_1)$ is zero.

Even if $\phi$ lifts to an element of $H^1(G_{F}, M)$, it still might not lift to an element of $\text{Sel}(M, \mathscr{W})$. To understand this failure, consider the following commutative diagram with exact rows:
\[\begin{tikzcd}[column sep=small]
 & H^1(G_F, M_1) \arrow{r} \arrow{d} & H^1(G_F, M) \arrow{r}\arrow{d} & \ker\left(\begin{array}{l} H^1(G_F, M_2)\\ \rightarrow H^2(G_F, M_1)\end{array}\right) \arrow{r} \arrow{d} &0 \\
0 \arrow{r} & \dfrac{\tresprod{v}H^1(G_{v}, M_1)}{\msW_1}  \arrow{r} &  \dfrac{\tresprod{v} H^1(G_v, M)}{\mathscr{W}} \arrow{r} & \dfrac{\tresprod{v} H^1(G_v, M_2)}{\msW_2 }. & \end{tikzcd}\]
The snake lemma then gives an exact sequence
\begin{align}
\label{eq:CT_obstr2}
0 \rightarrow  \Sel( M, \, \mathscr{W}) \,&\xrightarrow{\,\,\pi\,\,} \,\ker\Big(\Sel\big(M_2,\, \msW_2\big) \rightarrow \Sha^2(M_1)\Big)\\
&\rightarrow\, \coker\left(H^1(G_{F}, M_1) \rightarrow\dfrac{\tresprod{v}H^1(G_v, M_1)}{\msW_1}\right).\nonumber
\end{align}
If $\phi$ maps to zero in this cokernel, it is necessarily of the form $\pi(\phi')$ for some $\phi'$ in $\text{Sel}(M, \mathscr{W})$. This cokernel then gives the second obstruction to lifting $\phi$ to  $\text{Sel}(M, \mathscr{W})$.

The main observation that allows us to construct the Cassels--Tate pairing is that these obstructions can be packaged together using the dualities \eqref{eq:first_step_duality} and \eqref{eq:second_step_duality}. The first of these gives a pairing between $\Sha^2(M_1)$, where the first obstruction lies, and $\Sha^1(M_1^{\vee})$. And the second gives a pairing between the cokernel considered  above and $\Sel\, M_1^{\vee}/ \Sha^1(M_1^{\vee})$. We can thus account for both obstructions using two filtered pieces of the Selmer group 
$\Sel\, M_1^{\vee}.$ This is the basic framework the Cassels--Tate pairing uses.

\begin{proof}[Proof of Proposition \ref{prop:left_kernel}]
We now verify that the construction of Definition \ref{defn:CTP} correctly accounts for both obstructions mentioned above. Specifically, the proposition is a consequence of the following two subclaims:
\begin{enumerate}
\item Suppose $\phi$ lies in $\Sel\, M_2$. Then $\phi$ can be written as $\pi(\phi')$ for some $\phi'$ in $H^1(G_{F}, M)$ if and only if $\text{CTP}(\phi, \psi) = 0$ for all $\psi$ in $\Sha^1(M_1^{\vee})$.
\item If $\phi$ takes the form $\pi(\phi')$, then it lies in $\pi(\Sel\, M)$ if and only if $\text{CTP}(\phi, \psi) = 0$ for all $\psi$ in $\Sel\, M_1^{\vee}$.
\end{enumerate}
We start by verifying the first claim. Choose an element $f$ in $C^1(G_{F}, M)$ such that $\pi\circ f$ represents the class of $\phi$. Taking $\delta$ to be the connecting map of the sequence \eqref{eq:CT_obstr1}, we see that $\iota^{-1}(df)$ represents the class of $\delta(\phi)$ in $H^2(G_{F}, M_1)$. Comparing Definition \ref{defn:CTP}  with Tate's definition  \cite{Tate63}  of the pairing \eqref{eq:first_step_duality}, we see that the identity 
\begin{equation*}
\label{eq:CT_PT}
\text{CTP}(\phi, \psi) = \text{PTP}(\iota^{-1}(df), \,\psi) 
\end{equation*}
holds for all $\psi$  in $\Sha^1(M_1^{\vee})$.
Nondegeneracy of \eqref{eq:first_step_duality} now shows that  $\text{CTP}(\phi, \psi)$ is zero for all $\psi \in \Sha^1(M_1^{\vee})$ if and only if  $\delta(\phi)=0$. This  proves the claim.  

To establish the second claim, suppose $\overline{\phi}$ can be written in the form $\pi(\overline{\phi'})$. We then can write
\[\text{CTP}(\phi, \psi) = \sum_{v}\text{inv}_v\left( \iota^{-1} \left(\overline{\phi'}_v - \overline{\phi}_{v, M}\right) \,\cup_{M_1}\, \overline{\psi}_v\right).\]
From the construction in the snake lemma,   the final map of \eqref{eq:CT_obstr2} is induced by  
\[\phi \mapsto \left(\overline{\phi'}_v - \overline{\phi}_{v, M}\right)_{v \text{ of } F}.\]
 The result then follows from the non-degeneracy of \eqref{eq:second_step_duality}.
\end{proof}

\subsection{The duality identity and the right kernel of the Cassels--Tate pairing}
\label{ssec:du_id}
We found the left kernel of $\CTP_E$ by a straightforward generalization of the work in \cite{Tate63, Milne86, Flach90}. Finding the right kernel  takes a little more work. We start by stating the end goal.
\begin{prop}
\label{prop:right_kernel}
The right kernel of the pairing $\CTP_E$ of Definition \ref{defn:CTP} is $\iota^{\vee}(\Sel \, M^{\vee})$.
\end{prop}
This can be proved directly, or as a simple consequence of the following essential symmetry result.
\begin{thm}[Duality identity]
\label{thm:dual_id}
The pairing given in Definition \ref{defn:CTP} satisfies
\begin{equation}
\label{eq:sym}
\CTP_E(\phi, \psi) = \CTP_{E^{\vee}}(\psi, \beta(\phi)),
\end{equation}
for all $\phi$ in $\Sel  (M_2,\, \msW_2)$ and all $\psi$ in $\Sel(M_1^{\vee},\, \msW_1^{\perp})$, where $\beta\colon M_2 \isoarrow  M_2^{\vee\vee}$ is the standard evaluation isomorphism.
\end{thm}
\begin{proof}[Proof of Proposition \ref{prop:right_kernel}]
By Theorem \ref{thm:dual_id}, the left kernel of $\CTP_E$ equals the left kernel of $\CTP_{E^{\vee}}$, which is $\iota^{\vee}(\Sel\, M^{\vee})$ by Proposition \ref{prop:left_kernel}.
\end{proof}

A key ingredient  in the proof of the duality identity is the following alternative construction for the Cassels--Tate pairing. This construction will also appear in the direct proof of Proposition \ref{prop:right_kernel} to follow.
\begin{defn}
\label{defn:CTPb}
Take $E$ to be the exact sequence in $\SMod_F$ considered in Definition \ref{defn:CTP}. Suppose we are given $\phi$ in $\Sel\, M_2$ and $\psi$ in $\Sel\, M_1^{\vee}$. As before, choose cocycles $\overline{\phi} \in Z^1(G_{F}, M_2)$ and $\overline{\psi} \in Z^1(G_{F}, M_1^{\vee})$ representing this pair of Selmer elements.

Instead of choosing $\big(f, \,\epsilon,\, (\overline{\phi}_{v, M})_v\big)$ as before, choose $\big(g,\, \eta, \,(\overline{\psi}_{v,M^\vee})_v\big)$ as follows:
\begin{itemize} 
\item Choose a cochain $g\in C^1(G_{F},\,M^\vee)$ satisfying $\overline{\psi}=\iota^\vee \circ g;$
\item Choose a cochain  $\eta\in C^2(G_F,\, (\Fsep)^{\times})$ satisfying
\[ d\eta= \overline{\phi}\,\cup_{M_2} \,(\pi^{\vee})^{-1} (dg);\]
\item For each $v\in S$, choose  $\overline{\psi}_{v,M^\vee}$ in $Z^1(G_v,M^\vee)$ such that 
$\iota^\vee(\overline{\psi}_{v,M^\vee})=\overline{\psi}_v$
 and so that the tuple $(\overline{\psi}_{v,M^\vee})_{v\in S}$ represents a class in $\mathscr{W}^\perp$.
\end{itemize}

Then
\[\gamma_v^{\text{bis}} \,=\, \overline{\phi}_v \cup_{M_2} (\pi^{\vee})^{-1}\left(\overline{\psi}_{v,M^\vee}-g_v\right)\,-\,\eta_v\]
lies in $Z^2\left(G_v, (\Fsep)^{\times}\right)$ for each $v$, and we define the pairing by
\[\CTPb{E}(\phi, \psi) = \sum_{v \text{ of }F} \inv_v(\gamma^{\text{bis}}_v).\]
\end{defn}
 
\begin{lem} \label{lem:weak_symmetry}
The pairing $\CTPb{E}$ given in Definition \ref{defn:CTPb} is well-defined, and we have 
\begin{equation} \label{eq:alternative_C_T}
\CTPb{E} \,=\, \CTP_E.
\end{equation}
\end{lem}

\begin{proof}
Choose any tuple 
\[\left(\overline{\phi},\,\overline{\psi},\, g,\, \eta, \,(\overline{\psi}_{v,M^\vee})_{v\text{ of }F}\right)\]
for computing $\CTPb{E}(\phi, \psi)$. Defining the cocycle $\gamma^{\text{bis}}_v$ for each $v$ as in Definition \ref{defn:CTPb}, we will show
\begin{equation}
\label{eq:bis_lem}
\CTP_E(\phi, \psi) = \sum_v \inv_v(\gamma^{\text{bis}}_v).
\end{equation}
This will show that the pairing $\CTPb{E}$ is well-defined and that it satisfies \eqref{eq:alternative_C_T}. 

To begin, choose $f$ and $(\overline{\phi}_{v,M})_v$ as in Definition \ref{defn:CTP}. Since 
\begin{align*}
d(f \cup_M g)\,&=\,df \cup_Mg -f \cup_M dg\\
&=\,\iota^{-1}(df) \cup_{M_1} \overline{\psi}\,-\,\overline{\phi} \cup_{M_2} (\pi^{\vee})^{-1}(dg),
\end{align*}
the assignment $\epsilon = \eta+ f\cup_Mg$ satisfies the requirements of Definition \ref{defn:CTP}. Defining the cocycle $\gamma_v$ for each $v$ as in Definition \ref{defn:CTP}, we have
\begin{align*}
&\gamma_v \,=\,  (f_v-\overline{\phi}_{v,M})\cup_M \overline{\psi}_{v,M^\vee} \,-\, \epsilon_v \quad\text{and}\\
&\gamma^{\text{bis}}_v  \,=\, f_v \cup_M (\overline{\psi}_{v,M^\vee} - g_v) \,-\, \eta_v.
\end{align*}
Thus
\begin{align*}
\gamma_v - \gamma^{\text{bis}}_v \,&=\, (f_v-\overline{\phi}_{v,M})\cup_M \overline{\psi}_{v,M^\vee}\,+\,f_v \cup_M (g_v-\overline{\psi}_{v,M^\vee})\,-\, f_v\cup_Mg_v \\
&= \,- \overline{\phi}_{v,M}\cup_M \overline{\psi}_{v,M^\vee}.
\end{align*}
We then have
\begin{align*}
 \sum_v \inv_v(\gamma_v - \gamma_v^{\text{bis}})  &=\, -\sum_v \inv_v(\overline{\phi}_{v,M}\cup_M \overline{\psi}_{v,M^\vee}).
\end{align*}
Since $(\overline{\phi}_{v,M})_{v}$ represents a class in $\mathscr{W}$ while $(\overline{\psi}_{v,M^\vee})_{v}$ represents a class in $\mathscr{W}^\perp$, this last sum vanishes. This establishes \eqref{eq:bis_lem}, and the lemma is proved.
\end{proof}

\begin{proof}[Direct proof of Proposition \ref{prop:right_kernel}]
It is clear from the definition of $\CTPb{E}$ that the right kernel of this pairing contains $\iota^{\vee}(\Sel\, M^{\vee})$. By  Lemma \ref{lem:weak_symmetry}, the right kernel of $\CTP_E$  contains $\iota^{\vee}(\Sel\, M^{\vee})$ also.

We thus have a bilinear pairing
\begin{equation}
\label{eq:rk_cmpt}
\CTP_E \,\colon\, \frac{\Sel\, M_2}{\pi(\Sel\, M)} \,\times \, \frac{\Sel\, M_1^{\vee}}{\iota^{\vee}(\Sel\, M^{\vee})} \to \QQ/\Z
\end{equation}
whose left kernel is zero. Applying the same argument to $E^{\vee}$ defines a bilinear pairing
\begin{equation}
\label{eq:rk_cmpt_2}
\CTP_{E^{\vee}} \,\colon\,  \frac{\Sel\, M_1^{\vee}}{\iota^{\vee}(\Sel\, M^{\vee})} \,\times \, \frac{\Sel\, M_2}{\pi(\Sel\, M)} \to \QQ/\Z
\end{equation}
whose left kernel is zero.

From \eqref{eq:rk_cmpt}, we have
\[\#\frac{\Sel\, M_2}{\pi(\Sel\, M)} \,\le \, \# \frac{\Sel\, M_1^{\vee}}{\iota^{\vee}(\Sel\, M^{\vee})}.\]
The pairing \eqref{eq:rk_cmpt_2} gives the reverse inequality, forcing \eqref{eq:rk_cmpt} to be perfect.
\end{proof}

We now prove the duality identity, which finishes the proof of Theorem \ref{thm:CTP_intro} .

\begin{proof}[Proof of Theorem \ref{thm:dual_id}]
From Lemma \ref{lem:weak_symmetry}, it suffices to show that the pairing $\CTPb{E}$ given in Definition \ref{defn:CTPb} satisfies
\begin{equation}
\label{eq:ugly_cochain}
\CTPb{E}(\phi, \psi) = \CTP_{E^{\vee}}(\psi, \beta(\phi))
\end{equation}
for all $\phi$ in $\Sel  (M_2,\, \msW_2)$ and all $\psi$ in $\Sel(M_1^{\vee},\, \msW_1^{\perp})$.  

Choose any tuple 
\[\left(\overline{\phi},\,\overline{\psi},\, g,\, \eta, \,(\overline{\psi}_{v,M^\vee})_{v\text{ of }F}\right)\]
for computing $\CTPb{E}(\phi, \psi)$. Given any $\epsilon \in C^2(G_F, (\Fsep)^{\times})$ satisfying 
\[d\epsilon = (\pi^{\vee})^{-1}(dg) \cup_{M_2} \overline{\phi},\]
the tuple
\[\left(\overline{\psi},\,\beta \circ \overline{\phi},\, g,\, \epsilon, \,(\overline{\psi}_{v,M^\vee})_{v}\right)\]
gives an acceptable set of choices for calculating  $\CTP_{E^{\vee}}(\psi, \beta(\phi))$. Our goal is to construct an explicit such $\epsilon$. 

To ease notation, take $q$ as shorthand for $ (\pi^{\vee})^{-1}(dg)$.  Consider the cochain $h$ in $C^2(G_F,(\Fsep)^{\times})$ defined by
\[h(\sigma, \tau) = \left\langle \overline{\phi}(\sigma\tau) , \, q(\sigma, \tau) \right\rangle,\]
where $\left \langle ~,~\right \rangle$  is   the standard pairing between $M_2$ and $M_2^{\vee}$. We calculate
\begin{align*}
dh(\sigma, \tau, \rho) =\,& \left\langle\sigma\overline{\phi}(\tau\rho), \, \sigma q(\tau, \rho)\right\rangle\, -\, \left\langle\overline{\phi}(\sigma\tau\rho), \,q(\sigma\tau, \rho)\right\rangle \\
+\,&\left\langle  \overline{\phi}(\sigma\tau\rho), \,q(\sigma, \tau\rho)\right\rangle \,-\, \left\langle  \overline{\phi}(\sigma\tau), \,q(\sigma, \tau)\right\rangle\\
= \,&- \left\langle  \overline{\phi}(\sigma), \,\sigma q(\tau, \rho) \right\rangle\, + \left\langle \sigma\tau \overline{\phi}(\rho), \,q(\sigma, \tau) \right\rangle\\
=\,& \left(q \cup_{M_2}\beta(\overline{\phi}) \,-\,  \overline{\phi}\cup_{M_2} q \right)(\sigma, \tau, \rho)
\end{align*}
with the second equality following from bilinearity and the relations
\begin{align*}
&q(\sigma, \tau\rho) - q(\sigma\tau, \rho) = q(\sigma, \tau) - \sigma q(\tau, \rho)\quad\text{and}\\
&\overline{\phi}(\sigma\tau\rho) = \sigma \overline{\phi}(\tau\rho) + \overline{\phi}(\sigma) = \sigma\tau\overline{\phi}(\rho) + \overline{\phi}(\sigma\tau).
\end{align*}
We  may thus take $\epsilon = \eta + h$. Defining $\gamma^{\text{bis}}_v$ as in Definition \ref{defn:CTPb}, defining $\gamma_v$ for the pairing $\CTP_{E^{\vee}}(\psi, \beta(\phi))$ as in Definition \ref{defn:CTP}, and using the shorthand
\[b_v =  (\pi^{\vee})^{-1}\left(\overline{\psi}_{v,M^\vee}-g_v\right),\]
we have
\[\gamma^{\text{bis}}_v - \gamma_v \,\,=\,\, \overline{\phi}_v \cup_{M_2} b_v \,+ \,b_v \,\cup_{M_2} \,\overline{\psi}_v  \,+\, h_v.\]
Taking $a_v$ to be the cochain in $C^1(G_v, (\Fsep_v)^{\times})$ defined by
\[a_v(\sigma) = -\left\langle \overline{\phi}_v(\sigma) ,\,\, b_v(\sigma)\right\rangle\quad\text{for } \, \sigma\in G_v,\]
a similar calculation to the one given before gives
\[da_v = \gamma^{\text{bis}}_v - \gamma_v.\]
We then have
\[\CTPb{E}(\phi, \psi) - \CTP_{E^{\vee}}(\psi, \beta(\phi)) = \sum_v \inv_v(da_v) = 0,\]
giving the theorem.
\end{proof}

\section{The category $\SMod_F$}
\label{sec:category}
The naturality property of  Theorem \ref{thm:CTP_intro} (ii) gives a flexible approach for comparing the Cassels--Tate pairings associated to different short exact sequences in $\SMod_F$. This property also motivates the study of $\SMod_F$ as a category, which we turn to now.

\subsection{The quasi-abelian property}
\label{ssec:qab}
The category $\SMod_F$ is not abelian, as epic monic morphisms may fail to be isomorphisms. However, we will verify that the category is quasi-abelian. This gives us access to many of the same tools available for abelian categories.

We start with some terminology. Call a morphism $f: (M, \msW) \to (M', \msW')$ strictly monic if it is the kernel of some morphism in $\SMod_F$, and call it strictly epic if it is the cokernel of some morphism in $\SMod_F$. Equivalently, $f$ is strictly monic if it is injective in the category of $G_F$-modules and $f^{-1}(\msW') = \msW$, and it is strictly epic if it is surjective in the category of $G_F$-modules and $f(\msW) = \msW'$. It is easily checked from the kernel/cokernel definition that $f$ is strictly monic if and only if $f^{\vee}$ is strictly epic.

Consider a pair of morphisms
\[(M_1, \msW_1) \xrightarrow{\,\,\iota\,\,} (M, \msW) \xrightarrow{\,\, \pi\,\,} (M_2, \msW_2),\]
and suppose that this yields an exact sequence in the category of $G_F$-modules. In the category $\SMod_F$, we see that $\iota$ is the kernel of $\pi$ if and only if it is strictly monic; and we see that $\pi$ is the cokernel of $\iota$ if and only if it is strictly epic. From this, we find that  $\SMod_F$ has all kernels and cokernels. It also has finite products, and is hence preabelian.

There are many equivalent ways of characterizing which preabelian categories are quasi-abelian (see  e.g. \cite{Schn99, Buhl10}). For this work, we  use the notion of pullbacks and pushouts of  short exact sequences.

\begin{defn}
\label{defn:pullback}
Choose an exact sequence  
\[E \,=\, \left[0 \to M_1 \xrightarrow{\,\,\iota\,\,} M \xrightarrow{\,\,\pi\,\,}M_2 \to 0\right]\]
  in $\SMod_F$. Given a morphism $f \colon N_2 \to M_2$ in $\SMod_F$, we define the pullback of $E$ along $f$ as follows. First, take $N$ be the pullback of the morphisms $f$ and $\pi$. There is then a unique morphism $\iota_N \colon M_1 \to N$  whose composition with $f$ is $\iota$ and whose composition with $\pi$ is zero.  Writing $\pi_N$ for the morphism from $N$ to $N_2$, we then have a commutative diagram
\[\begin{tikzcd}
0 \arrow{r} & M_1 \arrow[d, equals] \arrow{r}{\iota_N} &  N \arrow{d} \arrow{r}{\pi_N} & N_2 \arrow{d}{f} \arrow{r} & 0\\
0 \arrow{r} & M_1 \arrow{r}{\iota} & M \arrow{r}{\pi} & M_2 \arrow{r} & 0
\end{tikzcd}\]
for which the right square is cartesian. We define the pullback of $E$ along $f$ to be the top row of this diagram.

Dually, given $g \colon M_1 \to N_1$, there is a unique commutative diagram
\[\begin{tikzcd}
0 \arrow{r} & M_1 \arrow{d}{g} \arrow{r}{\iota} &  M \arrow{d} \arrow{r}{\pi} & M_2 \arrow[d, equals] \arrow{r} & 0\\
0 \arrow{r} & N_1 \arrow{r} & N \arrow{r}& M_2 \arrow{r} & 0
\end{tikzcd}\]
for which the composition of the lower two maps is trivial and for which the left square is cocartesian. We define the pushout of $E$ along $g$ to be the bottom row of this diagram.
\end{defn}
\begin{prop}
\label{prop:is_quasi}
The category $\SMod_F$ is quasi-abelian. That is to say, given a short exact sequence $E$ in $\SMod_F$ as in Definition \ref{defn:pullback}, the pullback of $E$ along any morphism to $M_2$ is a short exact sequence; and dually, the pushout of $E$ along any morphism from $M_1$ is a short exact sequence.
\end{prop}
\begin{proof}
We will verify the condition for pullbacks. The proof for pushouts is similar.

Take $E$ and $f$ as in Definition \ref{defn:pullback}.  We will write the local conditions for $M$, $M_1$, $M_2$ as $\msW$, $\msW_1$, and $\msW_2$ respectively, and write the local conditions for $N$ and $N_2$ as $\msW'$ and $\msW'_2$ respectively. We will show that the pullback of $E$ along $f$ is short exact.

In the exact sequence
\[0 \to M_1 \xrightarrow{\,\,\iota_N\,\,} N  \xrightarrow{\,\,\pi_N\,\,} N_2 \to 0,\]
one can show that the morphism $\iota_N$ is the kernel of $\pi_N$ using the universal properties defining kernels and pullbacks. So we just need to show that $\pi_N$ is a strict epimorphism. As a map of $G_F$-modules, it is a pullback of a surjective map, and is hence surjective. So we just need to show that
\begin{equation}
\label{eq:quasi_epi}
\pi_N(\msW') = \msW'_2.
\end{equation}
Take $\Phi'_2$ in $\msW'_2$. We wish to show that there is $\Phi'$ in $\msW'$ satisfying $\pi_N(\Phi') = \Phi'_2$.

We have a short exact sequence
\begin{equation}
\label{eq:uglypullback}
0 \to N \to N_2 \oplus M \xrightarrow{\,-f \oplus \pi\,} M_2 \to 0
\end{equation}
in $\SMod_F$. A short portion of the associated long exact sequence sequence is
\begin{equation}
\label{eq:longshort}
\resprod_v H^1(G_v, N) \to \resprod_v H^1(G_v, N_2 \oplus M) \to \resprod_v H^1(G_v, M_2).
\end{equation}
 Since $\pi(\msW) = \msW_2$, there is some $\Phi \in \msW$ so that $\pi(\Phi)$ equals $f(\Phi'_2)$. From \eqref{eq:longshort}, $(\Phi'_2, \Phi)$ is the image of some $\Phi'$ in $\tresprod{v} H^1(G_v, N)$. This $\Phi'$ is in $\msW$ since the first map of \eqref{eq:uglypullback} is strictly monic, giving \eqref{eq:quasi_epi}. 
\end{proof}

We note that pullbacks and pushouts commute. Specifically, given $E$, $f$, and $g$ as above, we have a diagram
\begin{equation}
\label{eq:pb_po}
\begin{tikzcd}[row sep=4pt, column sep=9pt]
&M_1\arrow[dl] \arrow[dd, equal] \arrow[rr]&& N_{\text{pb}} \arrow[dl] \arrow[rr] \arrow[dd] & &  N_2 \arrow[dl, equals] \arrow[dd] \\
N_1 \arrow[rr, crossing over]&& N \arrow[rr, crossing over] & & N_2 \\
& M_1 \arrow[dl]  \arrow[rr]&& M\arrow[dl] \arrow[rr]  & & M_2 \arrow[dl, equals] \\
N_1\arrow[from=uu, crossing over, equal]\arrow[rr] && N_{\text{po}} \arrow[from=uu, crossing over] \arrow[rr] & & M_2, \arrow[from=uu, crossing over]\\
\end{tikzcd}
\end{equation}
where the front and back faces are pullbacks of exact sequences by $f$ in $\SMod_F$, and the top and bottom faces are pushouts by $g$. This can be seen by constructing the diagram's front and bottom faces and applying \cite[Proposition 3.1]{Buhl10} to the resultant diagram
\[\begin{tikzcd}
0 \arrow{r} & M_1 \arrow{r}\arrow{d}{g} & N_{\text{pb}} \arrow{d}\arrow{r} & N_2 \arrow{d}{f} \arrow{r} & 0\\
0 \arrow{r} & N_1 \arrow{r} & N_{\text{po}} \arrow{r} & M_2 \arrow{r} & 0.
\end{tikzcd}\]

Pullbacks and pushouts give a rich source of diagrams to which naturality can be applied. One of the most interesting cases of this concerns Baer sums. Whilst typically   defined in abelian categories, their definition still works in quasi-abelian categories, as was originally noted in \cite{Rich77}.

\begin{defn}[Baer sum]
\label{defn:baer}
Given short exact sequences
\begin{alignat*}{2}
&E_a\, &&=\, \left[0 \to M_1 \xrightarrow{\,\,\iota_a\,\,} M_a \xrightarrow{\,\,\pi_a\,\,} M_2 \to 0\right] \quad\text{and}\\
&E_b\,&&=\, \left[ 0 \to M_1 \xrightarrow{\,\,\iota_b\,\,} M_b\, \xrightarrow{\,\,\pi_b\,\,} M_2 \to 0\right]
\end{alignat*}
in $\SMod_F$, consider the commutative diagram with exact rows
\begin{equation}
\label{eq:def_Baer}
\begin{tikzcd}
& 0 \arrow{r} & M_1 \oplus M_1 \arrow{r}{\iota_a \oplus \iota_b} & M_a \oplus M_b \arrow{r}{\pi_a \oplus \pi_b}  & M_2 \oplus M_2  \arrow{r}  & 0 \\
&0 \arrow{r}& M_1 \oplus M_1 \arrow{r} \arrow[u, equals] \arrow{d}{\nabla}&N \arrow{r} \arrow{u} \arrow{d}& M_2 \arrow{u}{\Delta} \arrow[d, equals]\arrow{r} &0\\
& 0 \arrow{r}& M_1\arrow{r} & M \arrow{r}& M_2 \arrow{r}&0,
\end{tikzcd}
\end{equation}
defined so the central row is the pullback of the top row along the diagonal morphism $\Delta \colon M_2 \to M_2 \oplus M_2$, and the bottom row is the pushout of the central row along the codiagonal morphism $\nabla\colon  M_1 \oplus M_1 \to M_1$. We then define the Baer sum $E_a + E_b$ of $E_a$ and $E_b$ to be the final row of this diagram. Taking $\text{Ext}(M_2, M_1)$ to be the set of isomorphism classes of short exact sequences in $\SMod_F$, this set forms an abelian group under Baer sum.  
\end{defn}

\begin{prop}[Trilinearity]
\label{prop:Baer}
Given $M_1$ and $M_2$ in $\SMod_{F}$, the Cassels--Tate pairing satisfies
\[\CTP_{E_a + E_b} = \CTP_{E_a} + \CTP_{E_b}\]
for all $E_a, E_b$ in $\textup{Ext}\left(M_2, M_1\right)$.
\end{prop}

\begin{proof} 
Given $\phi$ in $\Sel\,M_2$ and $\psi$ in $\Sel\, M_1$, we have
\[\text{CTP}_{E_a \oplus E_b}\big( (\phi, \phi), \, (\psi,\psi)\big) = \text{CTP}_{E_a}(\phi, \psi) + \text{CTP}_{E_b}(\phi, \psi).\]
We can then derive the proposition by applying the naturality property of Theorem \ref{thm:CTP_intro} (ii) to the two morphisms of short exact sequences in \eqref{eq:def_Baer}.
\end{proof}

\subsection{The pairing for infinite modules}
\label{ssec:infinite}
One major point of this paper is that there is value in considering the Cassels--Tate pairing as a pairing on Selmer groups of finite Galois modules. That said, it is undeniably convenient to be able to define the pairing for infinite Galois modules as in \cite{Flach90}.

We will focus on $\Z_{\ell}$-modules rather than $\widehat{\Z}$-modules. This is a harmless decision, as a $\widehat{\Z}$-module breaks into $\ell$-adic components, and a choice for local conditions for the entire module splits as a collection of choices for local conditions of its $\ell$-adic components.
\begin{defn}
\label{defn:SModFinf}
Choose a global field $F$ and a rational prime $\ell$ not equal to the characteristic of $F$. Given a topological $\Z_{\ell}$-module $M$ with a continuous $\mathbb{Z}_{\ell}$-linear action of $G_F$, we call $M$ \emph{acceptable} if
\begin{itemize}
\item $M$ is Hausdorff,
\item $M$ contains an open compact $\Z_{\ell}[G_F]$-submodule $U$ such that both $U$ and the Pontryagin dual of $M/U$ are finitely generated $\Z_{\ell}$-modules, and
\item $M^{I_v} = M$ for all but finitely many places $v$ of $F$.
\end{itemize}
\color{red}

\color{black}

Given an acceptable $M$ and a closed subgroup $H$ of $G_F$, we define $H^1(H, M)$ using the continuous cochain cohomology of \cite{Tate76}. Our reference for these groups is \cite[II.7.1]{Neuk08}. We then take $\SMod_{F, \ell}^{\infty}$ to be the category whose objects are tuples $(M, \msW)$, where $M$ is an acceptable module and $\msW$ is a subgroup of $\prod_{v \text{ of } F} H^1(G_v, M)$ of the form
\[\msW_S \times \prod_{v \not \in S} H^1_{\text{ur}}(G_v, M),\]
for some finite set $S$  of places of $F$ and $\Z_{\ell}$-submodule $\msW_S$  of $\prod_{v \in S} H^1(G_v, M)$. The morphisms in ths category from $(M,\msW)$ to $(M',\msW')$ will be the $\Z_{\ell}[G_F]$-linear homomorphisms $f:M\rightarrow M'$ such that $f(\msW)\subseteq \msW'$.
\end{defn}

\begin{rmk}
Given an acceptable $M$ and open compact submodule $U$ of $M$ as in Definition \ref{defn:SModFinf}, we see that $M$ is topologically isomorphic to the inverse limit $\varprojlim_{n}M/l^nU$ of  discrete torsion $\mathbb{Z}_l[G_F]$-modules. Further, $M$ is the union of the compact open $G_F$-submodules $l^{-n}U$.

In addition, every acceptable module $M$ has finite $\ell$-rank in the sense of \cite[Definition 2.5]{MR2329311}. In particular, from \cite[Lemma 2.8]{MR2329311} we see that, after forgetting the $G_F$-action, $M$ is topologically isomorphic to a module of the form
\[\mathbb{Z}_{\ell}^{n_1} \oplus \mathbb{Q}_{\ell}^{n_2}  \oplus \big(\mathbb{Q}_{\ell}/\mathbb{Z}_{\ell}\big)^{n_3}\oplus (\textup{finite abelian }{\ell}\textup{-group}),\]
for some integers $n_1,n_2, n_3\geq 0$. The $\Z_{\ell}$-linear homomorphisms between such modules are easy to describe, and are automatically continuous. We note in particular that in any short exact sequence 
\[0 \to M_1 \xrightarrow{\,\,\iota\,\,} M\xrightarrow{\,\,\pi\,\,} M_2 \to 0\]
of acceptable modules, the map $\pi$ admits a continuous set-section, and the topology on $M_1$ is induced from the injection into $M$. 
\end{rmk}

\begin{defn}
Given an acceptable module $M$, we define its dual by
\[M^{\vee} = \Hom_{\text{cont}}(M,\, (F^s)^{\times})\]
with the compact-open topology, where $(F^s)^{\times}$ is given the discrete topology. This module is isomorphic to the Pontryagin dual of $M$ as a topological abelian group, and is acceptable.

As in Definition \ref{defn:loc_pair}, for each place $v$ of $F$,  composition of cup-product and the local invariant map gives a pairing  
\begin{equation}
\label{eq:infinite_Tate}
H^1(G_v, M) \otimes H^1(G_v, M^{\vee}) \to \QQ_{\ell}/\Z_{\ell}.
\end{equation}

\begin{prop}
\label{prop:i_T_nd}
The pairing \eqref{eq:infinite_Tate} is non-degenerate.
\end{prop}
The proof of this proposition relies on the following lemma.

\begin{lem}
Take an acceptable module $M$, and take  compact open $G_F$-submodules $U_1 \supseteq U_2 \supseteq \dots$ of $M$ with zero intersection. Then  $M$ is topologically isomorphic to the inverse limit $\varprojlim_n M/U_n$, and the natural projections define an isomorphism
\begin{equation}
\label{eq:bye276}
H^1(H, M) \isoarrow \varprojlim_{n} H^1(H, M/U_n)
\end{equation}
for any closed subgroup $H$ of $G_F$.
\end{lem}
\begin{proof}
That  $M \cong  \varprojlim_n M/U_n$ follows from the definition of an acceptable module.  

To show surjectivity of \eqref{eq:bye276}, take $(\phi_n)_n$ in $\varprojlim_{n} H^1(H, M/U_n)$, and represent each $\phi_n$ by a cocycle $\overline{\phi_n}$. Successively adjusting the $\overline{\phi_n}$ by   coboundaries if necessary, we may assume that the collection $(\overline{\phi_n})_n$ comprises a cocycle valued in the inverse limit  $\varprojlim_n M/U_n$. Since $M \cong  \varprojlim_n M/U_n$, the surjectivity of \eqref{eq:bye276} follows.

We now prove injectivity. Choose any compact open $G_F$-submodule $U$ of $M$ containing $U_1$. From  \cite[2.7.6]{Neuk08}, the natural projection defines an isomorphism from $H^1(H, U)$ to $\varprojlim_{n} H^1(H, U/U_n)$. 

We note that any sequence of submodules
\[M/U \supseteq A_1 \supseteq A_2 \supseteq \dots\]
must stabilize after finitely many steps, as follows from the fact that $(M/U)^{\vee}$ is a finitely generated $\Z_{\ell}$-module,  hence Noetherian. Consequently,  the map
\[H^0(H, M/U)  \to {\varprojlim}_n H^0(H, M/U) \,\big/\,\text{im}\big( H^0(H, M/U_n)\big)\]
is surjective, where the images are taken with respect to the natural projections $M/U_n \to M/U$. This, along with left-exactness of inverse limits, gives a commutative diagram
\[\begin{tikzcd}
H^0(H, M/U)    \arrow{r} \arrow[d, equals] &   H^1(H, U) \arrow{r} \arrow{d}{\sim} & H^1(H, M)\arrow{d} \\
H^0(H, M/U) \arrow{r}  &  \varprojlim_{n} H^1(H, U/U_n) \arrow{r}  & \varprojlim_{n} H^1(H, M/U_n) 
\end{tikzcd}\]
with exact rows. Taking a union over $U$ gives the result. 
\end{proof}

\begin{proof}[Proof of Proposition \ref{prop:i_T_nd}]
 Choose a nonzero element $\phi$ in $H^1(G_v, M)$. We wish to find $\psi$ in $H^1(G_v, M^{\vee})$ so the pairing between $\phi$ and $\psi$ is nonzero.

As is possible by Lemma \ref{prop:i_T_nd}, take an open compact $G_F$-submodule $V$  so that $\phi$ has nonzero image in $H^1(G_v, M/V)$. If there is $\psi_0$ in $H^1(G_v, (M/V)^{\vee})$ which has nonzero pairing with the image of $\phi$ in $H^1(G_v, M/V)$, then we may take $\psi$ equal to the image of $\psi_0$. So we may assume that $M$ is discrete.

With this assumption, choose a finite $G_F$-submodule $U$ of $M$ so $\phi$ is the image of some element $\phi_0$ in $H^1(G_v, U)$. Since $\phi$ is nonzero, $\phi_0$ is not in the image of the connecting map
\[H^0(G_v, M/U) \to H^1(G_v, U).\]
From local Tate duality, there is $\psi_0$ in $H^1(G_v, U^{\vee})$  that is orthogonal to this image but which is not orthogonal to $\phi_0$. From \cite[I.4.5]{Neuk08}, the image of $\psi_0$ under the connecting map to $H^2(G_v, (M/U)^{\vee})$ is orthogonal to $H^0(G_v, M/U)$. From local Tate duality and  \cite[2.7.6]{Neuk08}, this image must then be $0$, so $\psi_0$ may be lifted to $H^1(G_v, M^{\vee})$, giving the result.
\end{proof}

Choose an acceptable module $M$ and a compact open $G_F$-submodule $U$ of $M$. We make $H^1(G_v, M)$ into a topological group by taking the images of the $H^1(G_v, \ell^kU)$ with $k \ge 0$ to be a neighborhood basis of $0$.  The group $H^1(G_v, U)$ is a finitely generated $\Z_{\ell}$-module, so $H^1(G_v, M)$ is locally compact. We then find that the above pairing identifies $H^1(G_v, M^{\vee})$ with the Pontryagin dual to $H^1(G_v, M)$. We also see that, for any finite set of places $S$, the product $\prod_{v \in S} H^1(G_v, U)$ is a finitely generated $\Z_{\ell}$-module, so its closed subgroups are precisely those closed under multiplication by $\Z_{\ell}$.

From this and \cite[Corollary 2]{Morris77}, we find that the functor $(M, \msW) \mapsto (M^{\vee}, \msW^{\perp})$ defines a duality on the category $\SMod_{F, \ell}^{\infty}$, where the orthogonal complement is defined with respect to the sum of local pairings.
\end{defn}

As before, we find that $\SMod_{F, \ell}^{\infty}$ is quasi-abelian, and we may define strictly epic and monic morphisms and short exact sequences as before. Given $(M, \msW)$ in $\SMod_{F, \ell}^{\infty}$, we may define $\Sel( M, \msW)$ as the subgroup of $H^1(G_F, M)$ mapping into $\msW$, just as before.

\begin{defn}
\label{defn:infinite_pairing}
Given an exact sequence
\begin{equation}
\label{eq:example_exact}
E = \left[0 \to M_1 \xrightarrow{\,\,\iota\,\,} M\xrightarrow{\,\,\pi\,\,} M_2 \to 0\right]
\end{equation}
in $\SMod_{F, \ell}^{\infty}$, we define the pairing
\[\CTP_E\colon \Sel\, M_2\, \times\, \Sel\, M_1^{\vee}\to \QQ/\Z\]
using Definition \ref{defn:CTP}, where $Z^i$ and $C^j$ are reinterpreted as groups of continuous cocycles and cochains. As part of this definition, we choose a tuple
\[(\overline{\psi},\, \overline{\phi},\, f, \,(\overline{\phi}_{v, M})_v, \,\epsilon).\]

That $f$ may be chosen to be continuous follows from the fact that $\pi$ has a continuous section and that the topology of $M_1$ is induced by the injection into $M$. The group $H^3(G_F,(F^s)^\times)$ is torsion since $(F^s)^\times$ is discrete, and the class of $\iota^{-1}(df) \cup \overline{\psi}$ has order a power of $\ell$. Thus \cite[Corollary 4.18]{Milne86} allows us to choose $\epsilon$ exactly as in Definition \ref{defn:CTP}. 
The rest of the definition is straightforward.
\end{defn}
It is easily checked from this definition that this pairing has the same naturality properties as the pairing for finite modules. It is also easily checked that it still obeys the duality identity.  The remaining task is to calculate the kernels of this pairing.

\begin{prop}
The left kernel of the pairing defined above is $\pi(\Sel\, M)$, and its right kernel is $\iota^{\vee}(\Sel\,M^{\vee})$.
\end{prop}
\begin{proof}
From the duality identity, the result for right kernels follows from the result for left kernels. So we focus on calculating the left kernel.  From the definition, it is clear that this kernel contains $\pi(\Sel\,M)$. \color{black}

Conversely, suppose $\phi$ is in the left kernel of this pairing. From \eqref{eq:bye276}, we see that $\phi$ lies in $\pi(\Sel\, M)$ if, for every compact open $G_F$-submodule $V$ of $M$, the image of $\phi$ in $\Sel\, M_2/\pi(V)$ lies in $\pi(\Sel\, M/V)$. From naturality, we find that the image of $\phi$ in $\Sel \,M_2/\pi(V)$ is in the left kernel of the pairing corresponding to
\[0 \to M_1/\iota^{-1}(V) \to M/V \to M_2/\pi(V) \to 0.\]
From this, we may assume that $M$ is discrete without loss of generality.

Choose finite submodules $U_1 \subseteq U_2 \subseteq \dots$ of $M$ with union $M$ so $\phi$ is in the image of some $\phi_1$ in $\Sel\, U_1$. Then $\Sel\, M_1^{\vee}$ equals $\varprojlim_n \Sel \,(\iota^{-1}(U_i))^{\vee}$ by \eqref{eq:bye276}. Since $\Sel \,(\iota^{-1}(U_1))^{\vee}$ is finite, there is some $n$ so the image of $\Sel\, M_1^{\vee}$ in this group equals the image of $\Sel\,(\iota^{-1}(U))^{\vee}$ for $U = U_n$. Applying naturality twice, we find that the image of $\phi_1$ in $\Sel\, \pi(U)$ is in the left kernel of the pairing corresponding to
\[0 \to \iota^{-1}(U) \to U \to \pi(U) \to 0.\]
Applying Theorem \ref{thm:CTP_intro} then shows that $\phi$ is in the image of $\pi(\Sel\, U)$, and we are done.
\end{proof}
 
\begin{ex}
\label{ex:Flach}
Choose a torsion-free finitely-generated $\Z_{\ell}$-module $T$ with a continuous action of $G_F$, and choose a set of local conditions $\msW$ for $T \otimes \QQ_{\ell}$ so $\msW$ is a $\QQ_{\ell}$-vector space. Then the pairing we construct corresponding to the exact sequence
\[0 \to T \to T \otimes \QQ_{\ell} \to T \otimes \QQ_{\ell}/\Z_{\ell} \to  0\]
equals the pairing considered by Flach in \cite[Theorem 1]{Flach90}.
\end{ex}

\subsection{More general local conditions} \label{sec:general_local_conds}
Another direction in which to expand the category $\SMod_F$ is to allow for more general local conditions subgroups. Ideally, we would allow arbitrary closed subgroups for this definition. Unfortunately, given a surjection $\pi\colon M \to M_2$ of $G_F$-modules and a closed subgroup $\msW$ of $\tresprod{v} H^1(G_v, M)$, it is not necessarily true that $\pi(\msW)$ is closed in $\tresprod{v} H^1(G_v, M_2)$. We give an example of this behavior now.
\begin{ex}
\label{ex:bad_local}
Take $F = \QQ$, and take $\pi$ to be the unique surjection from $\mu_4$ to $\mu_2$. At each rational prime $p\neq 2$, choose $\phi_p$ in $H^1(G_p, \mu_4)$ which is ramified, but for which $2\phi_p$ is a nonzero unramified cocycle class. More concretely, under the standard identification
\[\QQ_p^{\times}/(\QQ_p^{\times})^4 \isoarrow H^1(G_p, \mu_4) \]
  we can take $\phi_v$ to be the image of $\alpha_p p^2$, where $\alpha_p$ is any non-square unit in $\Z_p^{\times}$.

Taking $ \msW =\resprod_{p \ne 2} \langle\phi_p \rangle$
we see that $\msW$ is closed in $\tresprod{v} H^1(G_v,\mu_4)$. But since elements in this restricted product are unramified at almost all places, we see that 
\[\pi(\msW) = \bigoplus_{v \ne 2, \infty} H^1_{\text{ur}}(G_v, \mu_2),\]
whose  closure   is the far larger $\prod_{v \ne 2, \infty}  H^1_{\text{ur}}(G_v, \mu_2)$.
\end{ex}
\begin{rmk}
Similar pathological local conditions can be constructed for the infinite Galois modules considered by Flach in \cite{Flach90}. In particular, it is possible to define local conditions for $T = \widehat{\Z}(1) \oplus \widehat{\Z}$ satisfying the requirements of \cite[p.114]{Flach90} for which the resultant Cassels--Tate pairing is not well defined. This oversight in Flach's definition is avoided in his main theorem, where additional requirements on the local conditions are imposed.
\end{rmk}

To avoid this pathology, we need to impose some extra property on our subgroups of local conditions. The following definition axiomatizes some properties this condition should have.
\begin{defn}
For each finite discrete $G_F$-module $M$ of order indivisible by the characteristic of $F$,   take $\mathscr{C}_M$ to be a set of subgroups of $\tresprod{v} H^1(G_v, M)$. We say a subgroup of this restricted product \emph{obeys condition $\mathscr{C}$} for $M$ if it lies in $\mathscr{C}_M$.

We say that condition $\mathscr{C}$ is \emph{Cassels--Tate ready} if the following conditions are satisfied:
\begin{enumerate}
\item If $\msW$ obeys condition $\mathscr{C}$, it is closed.
\item If $\msW$ and $\msW'$ obey condition $\mathscr{C}$ for $M$ and $M'$ respectively, then $\msW \oplus \msW'$ obeys condition $\mathscr{C}$ for $M \oplus M'$.
\item If $\msW$ obeys condition $\mathscr{C}$ for $M$, then $\msW^{\perp}$ obeys condition $\mathscr{C}$ for $M^{\vee}$.
\item Given any exact sequence
\[0 \to M_1 \xrightarrow{\,\,\iota\,\,} M \xrightarrow{\,\,\pi\,\,}M_2 \to 0\]
of $G_F$-modules and a local conditions subgroup $\msW$ for $M$ obeying condition $\mathscr{C}$, the groups $\iota^{-1}(\msW)$ and $\pi(\msW)$ obey condition $\mathscr{C}$ for $M_1$ and $M_2$ respectively. In particular, they are closed.
\item Given any $M$, take $\msW_{\textup{global}}$ to be the image of $H^1(G_F, M)$ in $\tresprod{v} H^1(G_v, M)$. Then, if $\msW$ obeys condition $\mathscr{C}$ for $M$, the subgroup
\[\msW_{\textup{global}} + \msW \subseteq \resprod_{v} H^1(G_v, M)\]
is closed.
\end{enumerate}
\end{defn}

Given a Cassels--Tate ready condition $\mathscr{C}$, we may redefine our category $\SMod_F$ so its objects are tuples $(M, \msW)$ where, as usual, $M$ is a finite discrete $G_F$-module, and where    $\msW$ is any local conditions subgroup for $M$ that obeys condition $\mathscr{C}$. This category is still quasi-abelian, the Cassels--Tate pairing can still be defined and has the kernels given in Theorem \ref{thm:CTP_intro}, and naturality and the duality identity still hold. Indeed, the only part of our theory that fails in this context is the direct proof of Proposition \ref{prop:right_kernel}, since this relies on the finiteness of the group $\Sel\, M_2/\pi(\Sel\, M)$.

The condition of being compact and open is easily seen to be Cassels--Tate ready. A more general condition that still works is the following:
\begin{defn}
Given $M$, call a subgroup $\msW$ of $\tresprod{v} H^1(G_v, M)$ \emph{suitable} if there are subgroups $\msW_0$ and $\msW_c$ of $\tresprod{v} H^1(G_v, M)$ satisfying
\begin{equation}
\label{eq:suit_decomp}
\msW \,=\, \msW_0 + \msW_c
\end{equation}
where $\msW_c$ is compact and $\msW_0$ is the image of an open subgroup of $\tresprod{v} H^1(G_v, M_0)$ for some $G_F$-submodule $M_0$ of $M$.
\end{defn}

We omit the proof of the following proposition.
\begin{prop}
The condition of being a suitable subgroup is Cassels--Tate ready.
\end{prop}

Note that the condition of being suitable is the strongest condition that is simultaneously Cassels--Tate ready and satisfied by all compact local conditions subgroups (equivalently by all open local conditions subgroups).

\section{Theta groups and the Poonen--Stoll class}
\label{sec:theta}
 
 In this section we turn to proving Theorem \ref{thm:poonen-stoll_into}, characterising the failure of Flach's Cassels--Tate pairing to be alternating for modules with alternating structure. As mentioned in the introduction, we will deduce this from a more general version, presented in Theorem \ref{thm:theta_main} below.  That result rests on the study of certain central extensions which we refer to as theta groups. 
 
\begin{defn}
\label{defn:theta}
Given $(M, \msW)$ in $\SMod_F$  or $\SMod_{F,\ell}^{\infty}$, a \emph{theta group} for $M$ will be a (typically non-abelian) topological group $\mcH$ with a continuous $G_F$-action fitting in a $G_F$-equivariant exact sequence
\begin{equation}
\label{eq:theta_seq2}
0 \to (F^s)^{\times} \xrightarrow{\,\,\,\,\,} \mcH \xrightarrow{\,\,\pi_{\mcH}\,\,} M \to 0,
\end{equation}
with $(F^s)^{\times}$ central in $\mcH$. We will assume that $\pi_{\mcH}$ is a continuous open map and that some open subgroup of $\mcH$ has preimage $\{1\}$ in $(F^s)^{\times}$. This assumption ensures in particular that $\pi_\mcH$ admits a continuous set-section, and that the discrete topology on $(F^s)^\times$ is induced by its inclusion in $\mcH$. In the case that $M$ is in $\SMod_F$, the assumption is equivalent to $\mcH$ being discrete.

We define the \emph{associated map}
\begin{equation} \label{eq:associated_theta_map}
f_{\mcH}: M \to M^{\vee}
\end{equation}
to this theta group to be the unique homomorphism satisfying
\[[g, h] \,=\, ghg^{-1}h^{-1} \,=\, \big\langle \pi_{\mcH}(g), \, f_{\mcH}(\pi_{\mcH}(h))\big\rangle\]
for all $g, h$ in $\mcH$, where $\left \langle~,~\right \rangle$ is the evaluation pairing on $M\times M^\vee$. The map $f_{\mcH}$ is well defined because \eqref{eq:theta_seq2} is a central extension, satisfies
$f_{\mcH}^{\vee} = -f_{\mcH}$,  and is continuous.
\end{defn}

We want to consider maps on cohomology  associated to the sequence \eqref{eq:associated_theta_map}. Since $\mcH$ is not assumed to be abelian, we take the time to set out the relevant notation and conventions. 

\begin{notat} \label{notat:theta_connecting}
Given a closed subgroup $H$ of $G_F$ and a non-negative integer $k$, we take $C^k(H, \mcH)$ to be the set of continuous functions from $H^k$ to $\mcH$. Following the convention of \cite[I, \S5.7]{MR1466966}, we define a coboundary operation
\[d \colon C^1(H, \mcH) \to C^2(H, \mcH)\]
by
\[db(\sigma, \tau) = b(\sigma) \cdot \sigma b(\tau) \cdot b(\sigma\tau)^{-1}\quad\text{for }\sigma,\tau \in H.\]
We then have a connecting map
\[q_{\mcH, H}\colon H^1(H, M) \to H^2(H, (F^s)^{\times})\]
associated to the sequence \eqref{eq:theta_seq2}. This is defined as follows: given $a \in Z^1(H, M)$ representing the class $\phi \in H^1(H, M)$, and given $b \in C^1(H, \mcH)$ satisfying
$\pi_{\mcH} \circ b\, =\, a,$
we see that $db$ is a cocycle valued in $(F^s)^{\times}$. Here $b$ can be chosen to be continuous since $\pi_\mcH$ admits a continuous section. We take $q_{\mcH, H}(\phi)$ to be the class of $db$ in $H^2(H, (F^s)^{\times})$.  The existence of this map depends on the fact that \eqref{eq:theta_seq2} is a central extension; see \cite{MR1466966} for more details.   

In contrast to the abelian case,   $q_{\mcH, H}$ is not typically a homomorphism. Instead, one has the identity  
\begin{equation}
\label{eq:Zark}
q_{\mcH, H}(\phi + \psi) - q_{\mcH, H}(\phi) - q_{\mcH, H}(\psi) = \phi \cup_M f_{\mcH}(\psi)\quad\text{for all } \phi, \psi \in H^1(H, M),
\end{equation}
first proved by Zarhin \cite{Zark74}. In the terminology of \cite[\S2.1]{PR12}, $q_{\mcH, H}$ is a quadratic map. 
\end{notat}

\begin{defn}
With the above setup in hand, we define the map
\[q_{ \text{loc-sum}}\colon \resprod_v H^1(G_v, M) \to \QQ/\Z\]
by setting,  for $(\phi_v)_v$ in $\tresprod{v} H^1(G_v, M)$,
\begin{equation} \label{eq:loc_sum_q_map}
q_{ \text{loc-sum}}((\phi_v)_v) = \sum_v \inv_v(q_{\mcH,G_v}(\phi_v)).
\end{equation}
We  say that the local conditions group $\msW$ for $M$ is \emph{isotropic with respect to $\mcH$} if  
\begin{equation} \label{eq: isotropic_quad_form_conds}
 q_{\text{loc-sum}}( \msW)=0.
 \end{equation}
\end{defn}

\begin{rmk}  
To see that the sum \eqref{eq:loc_sum_q_map} is well defined, take $S$ to be a finite set of places $v$ containing the archimedean places, the places dividing the order of $M$, the places where $M \ne M^{I_v}$, and the places where $\phi_v$ is ramified. For $v$ outside $S$, we see that $q_{\mcH, G_v}(\phi_v)$ is in the image of
$H^2(G_v/I_v,\mathcal{O}_{F_v^{\textup{ur}}}^\times)$
in $H^2(G_v, (F_v^s)^{\times})$, where $\mathcal{O}_{F_v^{\textup{ur}}}$ is the ring of integers of the maximal unramified extension of $F_v$. But this last group is trivial by \cite[Proposition 7.1.2]{Neuk08}, so
\[\inv_v(q_{\mcH,G_v}(\phi_v)) = 0\]
for $v$ outside $S$.
\end{rmk}

Take $\mcH$ to be a theta group for $M$.
Suppose we also have an exact sequence
\[E \,=\, \left[0 \to (M_1, \msW_1) \xrightarrow{\,\,\iota\,\,}(M, \msW) \xrightarrow{\,\,\pi\,\,} (M_2, \msW_2) \to 0\right]\]
in $\SMod_F$ or $\SMod_{F,\ell}^{\infty}$. Take $\mcH_1$ to be the preimage of $\iota(M_1)$ in $\mcH$ endowed with the subspace topology. We then have a commutative diagram with exact rows
\begin{equation}
\label{eq:PS_pullback}
\begin{tikzcd}
0 \ar{r} & (F^s)^{\times} \ar{r}\ar[d, equals] & \mcH_1 \ar{r}{\pi_{\mcH_1}}\ar{d} &M_1 \ar{r} \ar{d}{\iota}  & 0\\
0 \ar{r} & (F^s)^{\times} \ar{r} & \mcH \ar{r}{\pi_{\mcH}} & M \ar{r} & 0
\end{tikzcd}
\end{equation}
whose top row is the pullback of the bottom row along $\iota$. Since the topology on $M_1$ is induced from $M$, our assumptions on $\mcH$ ensure that $\pi_{\mcH_1}$ is open and continuous, and that some open subgroup of $\mcH_1$ has preimage $\{1\}$ in $(F^s)^\times$.

\begin{ass}
\label{ass:theta}$\,$ We assume that:
\begin{enumerate}
\item The group $\mcH_1$ is commutative. Equivalently,   $\iota(M_1)$ is isotropic with respect to the commutator pairing for $\mcH$.
\item We have
\[q_{\text{loc-sum}}(\iota(\msW_1)) = 0.\]
\item We have 
$\iota(\msW_1) \subseteq f_{\mcH}(\msW)^{\perp}.$
\end{enumerate}
\end{ass}

\begin{rmk} \label{rmk:theta_iso_equiv} 
The first assumption implies that there are unique $G_F$-equivariant homomorphisms $f_1$ and $f_2$ fitting in the commutative diagram
\begin{equation}
\label{eq:theta_morph}
\begin{tikzcd}
0 \ar{r} & M_1\arrow{d}{f_1} \arrow{r}{\iota}& M \arrow{r}{\pi} \arrow{d}{f_{\mcH}} & M_2 \arrow{d}{f_2} \ar{r} & 0\\
0 \ar{r} & M_2^{\vee} \arrow{r}{\pi^{\vee}} &M^{\vee} \arrow{r}{\iota^{\vee}} & M_1^{\vee}\ar{r} & 0.
\end{tikzcd}
\end{equation}
The third assumption is equivalent to the assumption that $f_1$ and $f_2$ define morphisms in $\SMod_F$ (or $\SMod_{F,\ell}^{\infty}$ as appropriate), and is weaker than the assumption that $f_{\mcH}$ defines a morphism in $\SMod_F$.

If the local conditions $\msW$ are isotropic for $\mcH$, then the second and third assumptions are automatically satisfied.
\end{rmk}

\begin{defn}
\label{defn:PS_class}
The abelian group $(F^s)^{\times}$ is divisible by all primes not equal to the characteristic of $F$, so we may choose a homomorphism $s \colon M_1 \to \mcH_1$ giving a continous section to   $\pi_{\mcH_1}$.   Indeed, if $M_1$ is finite this is clear. If $M_1$ is infinite, our assumptions ensure that the restriction of $\pi_{\mcH_1}$ to some open subgroup of $\mcH_1$ admits such a section, and we can then extend this to all of $M_1$ using  Zorn's lemma.
We represent this with the diagram
\begin{equation}
\label{eq:H1_section}
\begin{tikzcd}
0 \ar{r} &  (F^s)^{\times} \ar{r} &  \mcH_1 \ar[r, "\pi_{\mcH_1}" ']& M_1 \ar{r} \ar[l, bend right, "s" '] & 0.
\end{tikzcd}
\end{equation}
We can then define a cocycle
\begin{equation}
\label{eq:PS_rep}
\overline{\psi}_{\textup{PS}} \in Z^1(G_F, M_1^\vee) 
\end{equation}
 by the formula
\[\overline{\psi}_{\textup{PS}}(\sigma)(m) = \sigma s (\sigma^{-1} m) - s(m) \quad\text{for all } \sigma \in G_F, \, m \in M_1.\]
We call the class of $\overline{\psi}_{\textup{PS}}$ in $H^1(G_F, M_1^{\vee})$ the \emph{Poonen--Stoll class}, and we represent it with the notation  $\psi_{\textup{PS}}$.
\end{defn}

\begin{rmk}
If $M_1$ is a finite module,  there is a a canonical isomorphism
\begin{equation}
\text{Ext}^1_{G_F}(M_1, (F^s)^{\times}) \cong H^1(G_F, M_1^{\vee})
\end{equation}
arising from the local-global Ext spectral sequence \cite[Lemma 4.12]{Milne86}. A spectral sequence argument shows that, under this isomorphism, the class of the top row of \eqref{eq:PS_pullback} maps to $\psi_{\text{PS}}$.
\end{rmk}

\begin{prop}
\label{prop:PS_Sel}
Given $E$ and $\mcH$ as in Definition \ref{defn:PS_class} that satisfy Assumption \ref{ass:theta}, the associated Poonen--Stoll class $\psi_{\textup{PS}}$ lies in the Selmer group $\Sel\, M_1^{\vee}$.
\end{prop}
\begin{proof}
A cocycle calculation gives that the connecting map
\[\delta_v\colon H^1(G_v, M_1) \to H^2(G_v, (F^s)^{\times})\]
associated to the exact sequence \eqref{eq:H1_section} satisfies
\[\delta_v(\phi_{1v}) \,= \,\psi_{\text{PS},\, v} \cup_{M_1}  \phi_{1v}\quad\text{for all}\quad \phi_{1v} \in H^1(G_v, M_1). \]
Functoriality applied to  \eqref{eq:PS_pullback}  gives $\delta_v \,=\, q_{\mcH, G_v} \circ \iota,$
so   Assumption \ref{ass:theta} (2) gives
\[\sum_v \inv_v\left(\psi_{\text{PS},\,v} \cup_{M_1}  \phi_{1v}\right) = q_{\text{loc-sum}}((\iota(\phi_{1v}))_v) = 0\quad\text{for all}\quad (\phi_{1v})_v \in \msW_1.\]
It follows that $(\psi_{\text{PS}, \,v})_v$ lies in $\msW_1^{\perp}$ and thus that $\psi_{\text{PS}}$ lies in $\Sel\, M_1^{\vee}$.
\end{proof}

Having defined the Poonen--Stoll class we can state the main result of this section.

\begin{thm}
\label{thm:theta_main}
Suppose we are in the situation of Assumption \ref{ass:theta}. Let $\phi$ be any element of $\Sel\, M_2$, and choose $(\phi_{v, M})_v \in \msW$ projecting to the image of $\phi$ in $\msW_2$. We then have
\begin{equation}
\label{eq:theta_main}
\CTP_E\big(\phi,\,\,\psi_{\textup{PS}} - f_2(\phi)\big)\, =\, q_{\textup{loc-sum}}((\phi_{v, M})_v).
\end{equation}
\end{thm}
In particular, if the local conditions $\msW$ are isotropic with respect to $\mcH$, we have
\begin{equation}
\label{eq:theta_main_isot}
\CTP_E\big(\phi,\,\,  f_2(\phi)\big) \,=\, \CTP_E\big(\phi,\,\,  \psi_{\textup{PS}}\big)  \quad\text{for all }\,\,\phi \in \Sel\, M_2,
\end{equation}
and the pairing $\CTP_E\big(\text{--},\,\, f_2(\text{--})\big)$ is alternating if and only if $\psi_{\textup{PS}}$ lifts to $\Sel\, M^{\vee}$.

\begin{rmk}
The second part of Assumption \ref{ass:theta} guarantees that $\psi_{\text{PS}}$ lies in $\Sel\, M_1^{\vee}$, and the third part of this assumption guarantees that $f_2(\phi)$ lies in $\Sel\, M_1^{\vee}$. These assumptions together ensure that the left hand side of \eqref{eq:theta_main} is well-defined.

Moreover, as can be verified from \eqref{eq:Zark}, the second and third parts of Assumption \ref{ass:theta} are together equivalent to the single condition
\begin{equation}
\label{eq:theta_sole_ass}
q_{ \text{loc-sum}}((\phi_{1v} + \phi_v)_v) \,=\, q_{ \text{loc-sum}}((\phi_v)_v)\quad\text{for all}\quad (\phi_{1v})_v \in \iota(\msW_1),\,\, (\phi_{v})_v\in \msW.
\end{equation}
This alternative form of these assumptions suffices to show that the right hand side of \eqref{eq:theta_main} does not depend on the choice of $(\phi_{v, M})_v$.
\end{rmk}

\subsection{Proving Theorem \ref{thm:theta_main}}
\label{ssec:theta_proof}
We may organize the groups relevant to the proof of Theorem \ref{thm:theta_main} in a diagram
\begin{equation}
\label{eq:sd_theta}
\begin{tikzcd}[column sep = small]
  &(F^s)^{\times} \arrow[r,equal]\arrow[d, hook] &(F^s)^{\times}\arrow[d, hook] & \\
& \mathcal{H}_1 \arrow[d, left, "\pi_{\mcH_1}"] \arrow[r, hook] &  \mathcal{H} \arrow[d,"\pi_{\mcH}"] \arrow[r] & M_2\arrow[d, equal]\\
& M_1 \arrow{r}{\iota} \arrow[u, bend left, "s"] & M   \arrow{r}{\pi}& M_2  
\end{tikzcd}
\end{equation}
with exact rows and columns. We view $\mcH_1$ as a subgroup of $\mcH$ and $(F^s)^{\times}$ as a subgroup of $\mcH_1$, and we define $\overline{\psi}_{\text{PS}}$ from $s$ as in \eqref{eq:PS_rep}.

Throughout this section, $H$ will denote a closed subgroup of $G_F$, and all cup products are defined as in Notation \ref{notat:cup} with respect to the evaluation pairing on $M$, $M_1$, or $M_2$. We denote this pairing as $\left \langle~,~\right \rangle$.  We will use $\cdot$ to refer to the operation for the group $\mcH$ and will otherwise primarily use additive notation.

Theorem \ref{thm:theta_main} is built out of three basic cochain calculations. We start with the most difficult of the three.
\begin{lem}
\label{lem:ddb}
Choose a cochain $b$ in $C^1(H, \mcH)$, and take 
\[a = \pi_{\mcH} \circ b\quad\text{and}\quad  a_2 = \pi \circ a.\]
Suppose that $a_2$ is a cocycle. Then $db$ is valued in $\mcH_1$, and
\[d(db) = -a \cup (f_{\mcH} \circ da).\]
\end{lem}
\begin{rmk}
In the exact sequence
\[0 \to \mcH_1 \to \mcH \to M_2 \to 0,\]
the abelian subgroup $\mcH_1$ is not typically central in $\mcH$. Because of this, there is no satisfactory way to define a map
\[H^1(H, M_2) \to H^2(H, \mcH_1);\]
after all, per the above proposition, the usual definition of such a map need not send cocycles to cocycles.  However, as noted by Serre  \cite[I, \S5]{MR1466966}, given $\phi$ in $H^1(H, M_2)$, this definition  instead  gives a cocycle class $\delta \phi$ in $H^2(H, \,{}_{\phi} M_1)$, where ${}_{\phi} M_1$ is the Serre twist of $M_1$ by $\phi$. The proof we give of Lemma \ref{lem:ddb} is quite similar to the proof Serre gives of the fact that $\delta \phi$ is a cocycle class.
\end{rmk}
\begin{proof}
Since $da$ lies in $C^1(H, \iota(M_1))$, we see that $db$ takes values in $\mcH_1$. We have
\begin{alignat*}{3}
d(db)(\sigma, \tau, \rho) = &\big(\sigma b(\tau) &&\cdot \sigma \tau b(\rho) &&\cdot \sigma b(\tau \rho)^{-1}\big)\\ \cdot&\big( b(\sigma\tau) &&\cdot \sigma \tau b(\rho) &&\cdot b(\sigma \tau \rho)^{-1}\big)^{-1}\\
\cdot & \big( b(\sigma)&& \cdot \sigma b(\tau\rho) &&\cdot b(\sigma \tau \rho)^{-1}\big)\\
\cdot & \big( b(\sigma)&& \cdot \sigma  b(\tau) &&\cdot b(\sigma \tau)^{-1}\big)^{-1}.
\end{alignat*}
These four terms lie in $\mcH_1$ and hence commute. Rearranging this expression gives
\begin{alignat*}{2}
&\big(b(\sigma) \cdot \sigma b(\tau \rho)\cdot b(\sigma \tau \rho)^{-1}\big) &&\,\cdot\,\big(b(\sigma \tau \rho)\cdot \sigma \tau b(\rho)^{-1} \cdot b(\sigma\tau)^{-1}\big)\\
 \cdot\,& \big(b(\sigma \tau) \cdot \sigma b(\tau)^{-1} \cdot b(\sigma)^{-1}\big) &&\,\cdot\, \big(\sigma b(\tau) \cdot \sigma \tau b(\rho) \cdot \sigma b(\tau \rho)^{-1}\big).
\end{alignat*}
Cancelling $b(\sigma \tau\rho)^{-1}\cdot b(\sigma \tau \rho)$ and $b(\sigma \tau)^{-1} \cdot b(\sigma \tau)$ leaves
\[\big[b(\sigma), \,- \sigma db(\tau, \rho)\big] \,=\,-\big( a \cup (f_{\mcH} \circ da)\big)(\sigma, \tau, \rho),\]
as was claimed.
\end{proof}

The next basic cochain calculation is a simple generalization of Zarhin's identity \eqref{eq:Zark}.
\begin{lem}
\label{lem:Zark_2}
Choose cochains $b, b'$ in $C^1(H, \mcH)$, and take
\[a = \pi_{\mcH} \circ b \quad\text{and}\quad a' = \pi_{\mcH} \circ b'.\]
Suppose $da' = 0$. Then $db'$ is valued in $(F^s)^{\times}$, and
\[d(b \cdot b') = db \cdot db' \cdot \left(a' \,\cup \,(f_{\mcH} \circ a)\right).\]
\end{lem}
\begin{proof}
Since $da' = 0$, it is clear that $db'$ is valued in $(F^s)^{\times}$. We have
\[db(\sigma, \tau) = b(\sigma) \cdot b'(\sigma) \cdot \sigma b(\tau) \cdot \sigma b'(\tau) \cdot b'(\sigma \tau)^{-1} \cdot b(\sigma \tau)^{-1}.\]
This equals
\begin{align*}
&b(\sigma) \cdot \big[b'(\sigma), \,\sigma b(\tau)\big] \cdot \sigma b(\tau) \cdot b'(\sigma) \cdot \sigma b'(\tau) \cdot b'(\sigma \tau)^{-1} \cdot b(\sigma \tau)^{-1}\\
=\,&b(\sigma) \cdot (a' \cup (f_{\mcH} \circ a))(\sigma, \tau)\cdot  \sigma b(\tau) \cdot db(\sigma, \tau)  \cdot b(\sigma \tau)^{-1}.
\end{align*}
Since both $a' \cup (f_{\mcH} \circ a)$ and $db$ are valued in the center of $\mcH$, this equals 
\[db(\sigma, \tau) \cdot db'(\sigma, \tau) \cdot  (a' \cup (f_{\mcH} \circ a))(\sigma, \tau),\]
giving the claim.
\end{proof}

The final basic calculation gives a cochain analogue for a special case of the compatibility between   the Ext product and  the cup product. We abbreviate $\res_H (\overline{\psi}_{\textup{PS}})$ as $\overline{\psi}_{\textup{PS}, H}$.
\begin{lem}
\label{lem:dsf}
Choose $k \ge 0$, and suppose $h$ is in $C^k(H, M_1)$. Then
\[d(s \circ h) = s \circ dh \,+\, \overline{\psi}_{\textup{PS}, H} \cup (f_{\mcH} \circ h).\]
\end{lem}
\begin{proof}
From the definition in Notation \ref{notat:contcoch}, we find that 
\begin{alignat*}{2}
&d(s \circ h)(\sigma_1, \dots, \sigma_{k+1}) \, &&-\, s\left(dh(\sigma_1, \dots, \sigma_{k+1}) \right)\\
 = \,\,&\sigma_1\left(s \left(h(\sigma_2, \dots, \sigma_{k+1})\right)\right)  \,&&-\, s\left(\sigma_1 \left(h(\sigma_2, \dots, \sigma_{k+1})\right)\right)\\\
=\,\,& \Big\langle\overline{\psi}_{\text{PS}}(\sigma_1),\,\,\, \sigma_1h(\sigma_2, \dots, &&\sigma_{k+1})\Big\rangle,
\end{alignat*}
where $\sigma_1, \dots, \sigma_k$ are arbitrary elements of $H$. This gives the result.
\end{proof}

These three basic lemmas combine to give the following.
\begin{prop}
\label{prop:theta_compound}
Choose $b$ in $C^1(H, \mcH)$, and take
\[a = \pi_{\mcH} \circ b\quad\text{and}\quad  a_2 = \pi \circ a.\]
Suppose $da_2 = 0$. Then $da$ is valued in $\iota(M_1)$, and we may choose $A_1 \in C^2(H, M_1)$ so
\[\iota \circ A_1 = da.\]
\begin{enumerate}
\item The cochain $db\, -\, s \circ A_1$ is valued in $(F^s)^{\times}$, and
\[d\left(db\, -\, s \circ A_1\right) = \left(f_2 \circ a_2\, -\, \overline{\psi}_{\textup{PS}, H}\right) \cup A_1.\] 
\item Choose $b'$ in $C^1(H, \mcH)$, and take
\[a' = \pi_{\mcH} \circ b'.\]
Suppose that 
\[\pi \circ a' = a_2 \quad\text{and}\quad da' = 0.\]
There is then $a_1$ in $C^1(H, M_1)$ so
\[a - a' = \iota \circ a_1,\]
and we have
\[db' \, \equiv \, \left(db \,-\, s \circ A_1\right) \,+\,\left(f_2 \circ a_2 - \overline{\psi}_{\textup{PS}, H}\right) \cup a_1 \quad\textup{mod } B^2(H, (F^s)^{\times}).\]
\end{enumerate}
\end{prop}
\begin{proof}
For the first part, we note that $\pi_{\mcH_1} \circ db$ and $\pi_{\mcH_1} \circ s \circ A_1$ both equal $A_1$, so this difference lies in $C^2(H, (F^s)^{\times})$. Since $dA_1 = 0$, Lemma \ref{lem:dsf} and Lemma \ref{lem:ddb} give
\[d\left(db\, -\, s \circ A_1\right) \,=\, -\overline{\psi}_{\textup{PS}, H} \cup A_1\, -\, a \cup (f_{\mcH} \circ da).\]
Since $f_{\mcH}^{\vee} = -f_{\mcH}$, we have
\[a \cup (f_{\mcH} \circ da) \, =\, - (f_2\circ a_2) \cup A_1,\]
and the part follows.

For the second part, we note that $b \cdot (b')^{-1}$ is valued in $\mcH_1$ and projects to $a_1$ under $\pi_{\mcH_1}$. We have
\[d(b \cdot (b')^{-1}) \,\equiv \, d (s \circ a_1) \quad\text{mod } B^2(H, (F^s)^{\times}).\]
Since $da_1 = A_1$, Lemma \ref{lem:dsf} gives
\[d(b \cdot (b')^{-1}) \,\equiv \, s \circ A_1 \,+\,\overline{\psi}_{\textup{PS}, H} \cup a_1.\]
But Lemma \ref{lem:Zark_2} then gives
\begin{align*}
db\, =&\, db' \,+\, d(b \cdot (b')^{-1}) \,+\, a' \cup (f_{\mcH} \circ (a - a') )\\
\equiv&\, db' \,+\, s \circ A_1 \,+\, \overline{\psi}_{\textup{PS}, H} \cup a_1\,- \, (f_2 \circ a_2) \cup  a_1 \quad\text{mod } B^2(H, (F^s)^{\times}),
\end{align*}
and the proposition follows.
\end{proof}

\begin{proof}[Proof of Theorem \ref{thm:theta_main}]
We will calculate
\[\CTPb{E^{\vee}}\big(\psi_{\textup{PS}} - f_2(\phi),\,\, \beta(\phi)\big)\]
where $\beta\colon M_2 \to (M_2^{\vee})^{\vee}$ is the evaluation isomorphism. The result   then follows from the duality identity and Lemma \ref{lem:weak_symmetry}. In the infinite case, we note that both the cocycle definition of $\CTPb{E^{\vee}}$, and the resulting Lemma \ref{lem:weak_symmetry}, carry over to sequences in $\textup{SMod}_{F,\ell}^\infty$ without issue.

So choose $\overline{\phi}$ in $Z^1(G_F, M_2)$ representing the class of $\phi$, and choose $g$ in    $C^1(G_F, M)$ and $b$ in  $C^1(G_F, \mcH)$ satisfying
\[\overline{\phi} = \pi \circ g \quad\text{and}\quad g = \pi_{\mcH} \circ b.\] 
At each place, choose $\overline{\phi}_{v, M}$ representing $\phi_{v, M}$, and whose image in $M_2$ agrees with the restriction of $\overline{\phi}$. From the first part of Proposition \ref{prop:theta_compound}, we   have
\[d(db \,-\, s \circ \iota^{-1} \circ dg) =  (f_2 \circ \overline{\phi} \,-\, \overline{\psi}_{\textup{PS}}) \cup (\iota^{-1} \circ dg),\]
so that
\[s \circ \iota^{-1} \circ dg \, -\, db \in C^2(G_F, (F^s)^{\times})\]
is an acceptable choice for $\eta$ in the context of Definition \ref{defn:CTPb}. We then have
\[\CTPb{E^{\vee}}\big(\psi_{\textup{PS}} - f_2(\phi),\,\, \beta(\phi)\big) \,=\, \sum_v \inv_v(\gamma^{\text{bis}}_v)\]
where
\[\gamma^{\text{bis}}_v \,=\, \res_{G_v}\left( \overline{\psi}_{\textup{PS}} \,-\, f_2 \circ \overline{\phi} \right) \,\cup\, \left(\overline{\phi}_{v, M} - \res_{G_v}(g) \right) \,+\, \res_{G_v}\left(db \,-\, s \circ \iota^{-1} \circ dg\right).\] 
From proposition \ref{prop:theta_compound}, the cocycle $\gamma^{\text{bis}}_v$ is in the class $q_{\mcH,\, G_v}(\phi_{v, M})$, and we have
\[\CTPb{E^{\vee}}\big(\psi_{\textup{PS}} - f_2(\phi),\,\, \beta(\phi)\big)  = q_{\text{loc-sum}}((\phi_{v, M})_v) \]
as claimed.
\end{proof}

\subsection{Poonen--Stoll classes for Flach's pairing} \label{sec:poonen_stoll_flach}
We now have all the tools needed to prove Theorem \ref{thm:poonen-stoll_into}. 

Choose a prime $\ell$ not equal to the characteristic of $F$. Choose a free finitely-generated $\Z_{\ell}$ module $T$ with a continuous action of $G_F$. Take $V = T \otimes \QQ_{\ell}$, and decorate $V$ with local conditions as in Definition \ref{defn:SModFinf}. We assume these local conditions form a $\QQ_{\ell}$-vector space.

In addition, suppose there is an alternating homomorphism $\lambda: V \to V^{\vee}$ extending to a morphism in $\SMod_{F, {\ell}}^{\infty}$ under which $T$ is identified with its orthogonal complement. In this case, the Cassels--Tate pairing for
\[E= \left[0 \to T \to V \to V/T \to 0\right]\]
reduces to a pairing between $\Sel\, V/T$ and itself. Flach proved that this pairing is antisymmetric; we may recover this result using the duality identity and naturality.  
We are now concerned with the obstruction to this pairing being alternating. The pairing is automatically alternating if $\ell \ne 2$, so we focus on the case $\ell = 2$. 

Composing the evaluation pairing  on $V \otimes V^{\vee}$ with $\lambda$ defines an alternating pairing on $V$. We denote this by $\lambda$ also, and construct the class $c_{\lambda} \in H^1(G_v, V/T)$ from this pairing as in Definition \ref{defn:PS_intro}. The following is then the precise form of Theorem \ref{thm:poonen-stoll_into}.

\begin{thm} \label{thm:poonen-stoll_precise}
Given $T$, $E$, $\lambda$, and $c_{\lambda}$ as above, the element $c_{\lambda}$ lies in $\Sel\, V/T$, and we have
\[\CTP_E\left(\phi, \,\lambda(\phi)\right) = \CTP_E\left(\phi, \,\lambda(c_{\lambda})\right)\quad\text{for all } \,\phi \in \Sel \, V/T.\]
\end{thm}
\begin{proof}
Take $\mcH(V)$ to be the group with underlying $G_F$-set $(F^s)^\times \times V$, and group operation 
\[(\alpha,v)\cdot (\alpha',v')=\big(\alpha\alpha'\left \langle \tfrac{1}{2}v, \lambda(v')\right \rangle, v+v'\big),\]
where $\left \langle ~,~\right \rangle$ denotes the evaluation pairing $V\otimes V^\vee \rightarrow (F^s)^\times$. The natural inclusion/projection on the level of $G_F$-sets realises $\mcH(V)$ as a central extension 
\begin{equation}\label{divisible_theta_group}
0\longrightarrow (F^s)^\times \longrightarrow \mcH(V)\longrightarrow V\longrightarrow 0,\end{equation}
whose associated commutator pairing  is given by 
\begin{equation}\label{eq:divisible_theta_group_pairing}
(v,v')\longmapsto \left \langle v,\lambda(v')\right \rangle ~~\quad v,v'\in V. 
\end{equation}
Our assumptions imply that $\mcH(V)$ gives a theta group for $V$ in the sense of Definition \ref{defn:theta}. Further, we see from Remark \ref{rmk:theta_iso_equiv} that the conditions of Assumption  \ref{ass:theta} are satisfied with respect to the sequence 
\[E= \left[0 \to T \to V \to V/T \to 0\right].\]
Let $\psi_{\textup{PS}}\in H^1(G_F,T^\vee)$ be the resulting Poonen--Stoll class. A direct computation shows that the pullback of \eqref{divisible_theta_group} along the inclusion of $T$ into $V$ agrees with the pullback of the sequence \eqref{eq:PS_exact_Flach_seq} along the projection  $T\rightarrow T/2T$. Consequently,   after identifying $H^1(G_F,T^\vee)$ with $H^1(G_F,V/T)$ via $\lambda$, we see that $c_\lambda$ agrees with $\psi_{\textup{PS}}$. The result now follows from Theorem \ref{thm:theta_main}.
\end{proof}

\begin{rmk}
In the case that $T$ is the Tate module for a principally polarized variety over a global field of characteristic other than $2$, our definition for $c_{\lambda}$ agrees with the definition of this cocycle class by Poonen and Stoll \cite[Theorem 5]{Poon99}. This can be proved using \cite[Remark 3.3]{PR11} and the diagram \cite[(16)]{PR11}.

This leaves the case of abelian varieties over characteristic $2$, which the geometric methods of Poonen and Stoll can handle but which our methods cannot. To deal with these examples, we would need to expand our theory to handle finite group schemes whose order is divisible by the characteristic of $F$. We will return to this in future work \cite{MoSm21b}.
\end{rmk}

\subsection{Theta groups for finite modules with antisymmetric structure}
\label{ssec:div_theta}

The construction of a theta group in the proof of Theorem \ref{thm:poonen-stoll_precise} made use of the fact that the starting module was a $\QQ_2$-vector space. For applications, it will be useful to have a generalization of this construction that applies to finite modules. As we will see in Section \ref{sec:theta_line_bundle}, this construction recovers Mumford's construction of theta groups associated to symmetric line bundles on abelian varieties over fairly general fields.  

\begin{defn}
\label{defn:div_theta}
Fix any field $F$ of characteristic not equal to $2$, fix a separable closure $\Fsep$ of $F$, and fix a discrete $G_F = \Gal(\Fsep/F)$-module $M$. We assume that $M[2]$ is finite.

Choose a finite $G_F$-submodule $M_0$ of $M$, and take 
$\lambda \colon M \to M/M_0$
to be the standard projection. For $k \ge 0$, we take 
$2^k \lambda\colon M \to M/M_0$
to be the composition of multiplication by $2^k$ with $\lambda$, and denote by  $M[2^k \lambda]$ the kernel of this map.

We assume that the final map in the exact sequence
\[0 \to M[2] \to M[2\lambda] \xrightarrow{\,\,\cdot 2\,\,} M[\lambda]  \]
is surjective, and we suppose that we have been given
\begin{itemize}
\item an antisymmetric $G_F$-equivariant pairing
\[P_1: M[2\lambda] \otimes M[2\lambda] \to (\Fsep)^{\times}\,\, \text{ and}\]
\item a $G_F$-equivariant map $e: M[2] \to \pm 1$ satisfying
\[e(x+y) e(x)^{-1}e(y)^{-1} = P_1(x, y) \quad\text{for all }\, x, y \in M[2].\]
\end{itemize}
Under these circumstances, there is a unique pairing
\[P_0: M[\lambda] \otimes M[\lambda] \rightarrow (\Fsep)^{\times}\]
satisfying
\[P_0(2m, 2n) = P_1(m, n)^2 \quad\text{for all } \, m, n \in M[2\lambda].\]
This pairing is alternating and $G_F$-equivariant.

Following the construction of \cite[Section 2.1]{PR11}, we define a group $U$ with underlying  set 
$(\Fsep)^{\times} \times M[2\lambda],$
 equipped with the same associated $G_F$-action, but group operation
\[(\alpha, m) \cdot (\alpha', m') \,=\,  \big(\alpha \alpha' P_1(m, m'),\, m + m'\big).\]
For any $(\alpha, m)$ and $(\alpha', m')$ in this group, we calculate
\[(\alpha, m)^{-1} = \big(\alpha^{-1}P_1(m, m), \,-m\big) \quad\text{and}\]
\[ (\alpha, m) (\alpha', m') (\alpha, m)^{-1} = \big(\alpha'  P_1(m, m')^2, \,m'\big).\]
From these basic properties, we can verify that
\[N=\big\{ (e(m),\, m) \in U \,:\,\, m \in M[2]\big\}\]
is a normal subgroup of $U$ stable under the action of $G_F$.  We then define
\[\mcalH(M[\lambda]) \, = \, U/N.\]
We have the following commutative diagram with exact rows:
\begin{equation}
\label{eq:thet_def}
\begin{tikzcd}
0 \arrow{r} & (\Fsep)^{\times} \arrow{r}\arrow[d, equals] & U \arrow[d] \arrow{r} &M[2\lambda]\arrow[d]  \arrow{r} &0 \\
0 \arrow{r} & (\Fsep)^{\times} \arrow{r} & \mcalH(M[\lambda]) \arrow{r} & M[\lambda]\arrow{r} &0.
\end{tikzcd}
\end{equation}
\end{defn}

The groups $U$ and $\mcalH(M[\lambda])$ are typically non-abelian, and we can consider the commutator pairings on either of the rows of \eqref{eq:thet_def}. From the calculation
\[(1, m) \cdot (1, m') \cdot (1, m)^{-1} \cdot (1, m')^{-1} = \big(P_1(m, m')^2, 0\big),\]
we see that this pairing takes the form $(m, \,m') \mapsto P_1(m, m')^2$ for the top row. On the bottom row, this pairing is then given by $(m, \,m') \mapsto P_0(m, m')$. 

\begin{ex}
\label{ex:theta_div}
Suppose now that $F$ is a global field, and choose $M$, $\lambda$, $e$, and $P_1$ as in the previous proposition. We assume the characteristic of $F$ does not divide $2\cdot \# M[\lambda]$.

The second row of \eqref{eq:thet_def} then exhibits $\mcH = \mcH(M[\lambda])$ as a theta group of $M[\lambda]$. The associated map
\[f_{\mcH} \colon M[\lambda] \to M[\lambda]^{\vee}\]
is determined by
\[P_0(m, m') = \langle m, \,f_{\mcH}(m')\rangle\]
for all $m, m'$ in $M[\lambda]$, where the pairing $\langle\,,\,\rangle$ is the standard evaluation pairing for $M$.

Now suppose we have chosen groups of local conditions $\msW_0$, $\msW_1$ so that
\[\left(M[2\lambda], \, \msW_1\right) \xrightarrow{\,\,\,\cdot 2\,\,\,}  \left(M[\lambda], \,\msW_0\right)\]
is a strictly epic morphism in $\SMod_F$. Suppose further that we have
\begin{equation}
\label{eq:P1isot}
\sum_v\inv_v(\phi_v \cup_{P_1} \phi_v) =0
\end{equation}
for all $(\phi_v)_v \in \msW_1$. Then the group of local conditions $\msW_0$ for $M[\lambda]$ is isotropic with respect to $\mcH$.  Indeed, per \cite[Corollary 2.8]{PR11},  the connecting map  associated to the top row of \eqref{eq:thet_def} sends the class of a $1$-cocycle $\phi$ to 
$\phi \cup_{P_1} \phi$.

The isotropy condition \eqref{eq:P1isot} is weaker than assuming that the $G_F$-homomorphism $M[2\lambda] \to M[2\lambda]^{\vee}$ corresponding to $P_1$ extends to a morphism
\[\left(M[2\lambda], \, \msW_1\right) \xrightarrow{\quad} \left(M[2\lambda], \, \msW_1\right)^{\vee}\]
in $\SMod_F$. This condition is satisfied for a wide variety of self-dual $G_F$-modules in $\SMod_F$, so Definition \ref{defn:div_theta} gives a robust method for constructing theta groups meeting the assumptions of Theorem \ref{thm:theta_main}.
\end{ex}

We  mention the following additional property enjoyed by the theta groups of Example \ref{ex:theta_div}.
\begin{prop}
Given $\mcH$ as in Example \ref{ex:theta_div} and a closed subgroup $H$ of $G_F$, the connecting map
\[q_{\mcH,H} \colon H^1(H, M[\lambda]) \to H^2(H, (\Fsep)^{\times})\]
is a quadratic form.
\end{prop}
\begin{proof}
From \eqref{eq:Zark}, we have
\[q_{\mcH,H}(\phi + \psi) - q_{\mcH,H}(\phi) - q_{\mcH,H}(\psi) = \phi \cup_{P_0} \psi\]
for all $\phi, \psi$ in $H^1(H, M[\lambda])$. So we just need to show that
$q(a\phi) = a^2q(\phi)$
for all integers $a$. Per a standard argument (see \cite[Remark 2.1]{PR12}), this reduces to showing that
\begin{equation}
\label{eq:min_no_change}
q(-\phi) = q(\phi)
\end{equation}
for all $\phi$ in $H^1(H, M[\lambda])$. But taking $\gamma\colon \mcH \to \mcH$ to be the equivariant automorphism
\[(\alpha, m)\,\text{ mod } N\, \xmapsto{\quad\,\,}\, (\alpha, -m)\,\text{ mod } N,\]
we see that $\gamma$ fits in the diagram
\begin{equation}   \label{eqn:symm_theta}
\begin{tikzcd}
0 \arrow{r} & (\Fsep)^{\times}   \arrow{r}\arrow[d, equal] & \mcalH \arrow[d, "\gamma"] \arrow{r} &M[\lambda] \arrow[d, "  -1"]  \arrow{r} &0 \\
0 \arrow{r} & (\Fsep)^{\times}  \arrow{r} & \mcalH \arrow{r} & M[\lambda] \arrow{r} &0.
\end{tikzcd}
\end{equation}
Functoriality of connecting maps then gives \eqref{eq:min_no_change}.
\end{proof}

\begin{rmk}
In general, given a theta group $\mcH$ for the $G_F$-module $M$, the existence of an equivariant automorphism $\gamma$ fitting in the diagram obtained by replacing $M[\lambda]$ by $M$ in \eqref{eqn:symm_theta} is enough to show that $q_{\mcH}$ is a quadratic form. By analogy with the case of abelian varieties,  theta groups admitting such an automorphism might reasonably be called symmetric.
\end{rmk}

\subsection{Recovering Mumford's theta groups} \label{sec:theta_line_bundle}
Our use of theta groups was inspired by the theory of theta groups appearing in the theory of abelian varieties. In particular, Definition \ref{defn:div_theta} was designed to generalize Mumford's construction of theta groups associated to a symmetric line bundle on an abelian variety \cite{Mumf66}. The final goal of this section is to show that Definition \ref{defn:div_theta} recovers Mumford's construction.

To begin, we recall Mumford's definition of a theta group.
\begin{defn}
\label{defn:vanilla_theta}
Let $F$  be a field of characteristic other than $2$ and let $A$  be an abelian variety over  $F$. Let $L$ be a symmetric ample line bundle over $A$,  defined over $F$, and of degree coprime to the characteristic of $F$. Denote the dual abelian variety by $A^{\vee}$, and take $\lambda: A \rightarrow A^{\vee}$ to be the polarization associated to $L$. Write $A_{\Fsep}$ for the base-change of $A$  to $\Fsep$. By a slight abuse of notation, we write $L$ also for the base-change of $L$ to a line bundle on $A_{\Fsep}$. For $x \in A(\Fsep)$, write $\tau_x: A_{\Fbar} \rightarrow A_{\Fsep}$ for the translation by $x$ map. 

The theta group associated to the above data is defined to be the collection of pairs $(x,\phi)$ where $x$ is in $A[\lambda]$, and $\phi$ is an isomorphism (over $\Fbar$) from $L$ to $\tau^*_xL$:
\[\mcalH(L) = \big\{ (x, \phi)\,:\,\, x \in A[\lambda],\,\, \phi: L \xrightarrow{\,\,\sim\,\,} \tau^*_x L\big\}.\]
 The group structure is given by 
 \[(x,\phi) \cdot (x',\phi')=(x+x', \,\tau_{x'}^*(\phi)\circ \phi'),\]
 and the $G_F$-action is the diagonal one.  
 
Projection onto the first coordinate yields a short exact sequence of $G_F$-groups 
 \begin{equation}
\label{eq:thet_realdef}
1 \rightarrow \Fstar \rightarrow \mcalH(L) \rightarrow A[\lambda] \rightarrow 1,
\end{equation}
realizing the theta group as a central extension of $A[\lambda]$ by $\Fstar$. The associated commutator pairing 
\[e^{L}:A[\lambda]\times A[\lambda] \to \Fstar\]
is the Weil pairing associated to $L$.  
\end{defn}

\begin{prop}
\label{prop:degeom_theta}
With the notation above, write $M = A[\lambda]$, $M_0 = A[2\lambda]$,  and take $P_1$ to be the reciprocal  of the Weil pairing associated to the line bundle $L^{ 2}$. Let $e$ be the quadratic form
\[e^L_*: A[2] \rightarrow \pm 1\]
defined in \cite[p. 304]{Mumf66}. Let $\mathcal{H}(A[\lambda])$ be the theta group associated to this data via Definition \ref{defn:div_theta}.

Then there is a canonical isomorphism of  $G_F$-groups
\[\eta:  \mcalH(A[\lambda]) \rightarrow \mcalH(L)\]
   fitting into a commutative diagram
\[
\begin{tikzcd}
0 \arrow{r} & (\Fsep)^{\times}\arrow{r}\arrow[d, equals] &\mcalH(A[\lambda]) \arrow[d, "\eta"] \arrow{r} &A[\lambda] \arrow[d, equals]  \arrow{r} &0 \\
0 \arrow{r} & (\Fsep)^{\times} \arrow{r} &  \mcalH(L) \arrow{r} & A[\lambda] \arrow{r} &0
\end{tikzcd}
\]
with top row \eqref{eq:thet_def} and bottom row \eqref{eq:thet_realdef}.
\end{prop}

\begin{proof}
Since $L$ is a symmetric line bundle, there are canonical homomorphisms
\[\delta_{-1}: \mcalH(L^2) \rightarrow \mcalH(L^2)\quad\text{and}\quad \eta_2: \mcalH(L^2)\rightarrow \mcalH(L)\]
defined by Mumford \cite[p. 308]{Mumf66} and fitting into the commutative diagrams
\[
\begin{tikzcd}
0 \arrow{r} & \Fstar \arrow{r}\arrow[d, equals] & \mcalH(L^2) \arrow[d, "\delta_{-1}"]\arrow{r} &A[2\lambda] \arrow[d, " -1"]  \arrow{r} &0 \\
0 \arrow{r} & \Fstar\arrow{r} & \mcalH(L^2) \arrow{r} &A[2\lambda]  \arrow{r} &0
\end{tikzcd}
\]
and
\[
\begin{tikzcd}
0 \arrow{r} & \Fstar \arrow{r}\arrow[d, "x \mapsto x^2"] & \mcalH(L^2) \arrow[d, "\eta_2"]\arrow{r} &A[2\lambda] \arrow[d, "2"]  \arrow{r} &0 \\
0 \arrow{r} & \Fstar \arrow{r} & \mcalH(L) \arrow{r} &A[\lambda]  \arrow{r} &0,
\end{tikzcd}
\]
respectively. Here the appearance of $A[2\lambda]$ is explained by \cite[Proposition 4, p. 310]{Mumf66}.

Given $x$ in $A[2\lambda]$ and a lift $x'$ of $x$ to $\mathcal{H}(L^2)$,  we see that $\delta_{-1}(x')x'$ lies in $\Fstar$. Shifting $x'$ by an element of $\Fstar$ if necessary, it follows that we can always  find some $\psi(x)$ in $\mcalH(L^2)$ projecting to $x$ that satisfies $\psi(x)^{-1} = \delta_{-1}(\psi(x))$. This element is determined up to sign.  In particular, $\eta' = \eta_2 \circ \psi$ is a well defined map from $A[2\lambda]$ to $\mathcal{H}(L)$ and sits in a commutative diagram
\[
\begin{tikzcd}
& & &A[2\lambda] \arrow[d, " 2"] \arrow[ld, "\eta'"]   & \\
0 \arrow{r} & \Fstar \arrow{r} & \mcalH(L) \arrow{r} &A[\lambda]  \arrow{r} &0.
\end{tikzcd}
\]
For $x, y \in A[2\lambda]$ we compute
\[\delta_{-1}\big(\psi(x)\psi(y)\big) = \psi(x)^{-1}\psi(y)^{-1} = [\psi(x)^{-1},\psi(y)^{-1}]  \cdot \big(\psi(x)\psi(y)\big)^{-1}\]
\[= e^{L^2}(x, y) \big(\psi(x)\psi(y)\big)^{-1}.\]
Correcting for this, we see that the two choices for $\psi(x+y)$ are
\begin{equation}  \label{eq:computation_of_psi}
\psi(x + y) = \pm \frac{1}{\sqrt{e^{L^2}(x, y)}}  \psi(x)\psi(y),
\end{equation}
so we have
\begin{equation*}
\eta'(x + y) = P_1(x, y) \eta'(x) \eta'(y).
\end{equation*}
This agrees with  multiplication in $U$ as in \eqref{eq:thet_def}, and we have a commutative diagram

\[
\begin{tikzcd}
0 \arrow{r} & \Fstar \arrow{r}\arrow[d, equals] & U \arrow[d, "\eta"] \arrow{r} &A[2\lambda] \arrow[d, " 2"]  \arrow{r} &0 \\
0 \arrow{r} & \Fstar \arrow{r} & \mcalH(L) \arrow{r} &A[\lambda]  \arrow{r} &0,
\end{tikzcd}
\]
where the central map sends $(\zeta, x)$ in $\Fstar \times A[2\lambda]$ to $\zeta   \eta'(x)$.

It follows formally from the diagram that $\eta$ is surjective, so to prove the proposition we just need to verify that the kernel of this map is the set of $\big(e^L_*(x), x\big)$, with $x$ ranging through the $2$-torsion points. We start by noting that, given $(\zeta,x)$ is in the kernel of $\eta$, $x$ must lie in $A[2]$. Equation \eqref{eq:computation_of_psi}  then shows that we may take $\psi(x)$ to have order $2$ in $\mathcal{H}(L^2)$, and it thus follows from \cite[Proposition 6]{Mumf66}  that, inside $\Fstar$, we have
\[\eta'(x)=e_*^L(x).\]
Thus,  $(\zeta,x)$ is in the kernel of $\eta$ if and only if $\zeta= e_*^L(x)$. The result follows.
\end{proof}

\section{Examples} 
\label{section:apps}
\label{sec:apps}

We end by surveying some instances of pairings in the literature which can be interpreted and studied using our generalized theory of the Cassels--Tate pairing.

\subsection{The Cassels--Tate pairing for isogeny Selmer groups} 
\label{ssec:isogeny}
We start with a well-studied generalization of Example \ref{ex:classical} to isogenies. Choose an isogeny
$\varphi\colon A \to B$
of abelian varieties over $F$. We assume the degree of this isogeny is indivisible by the characteristic of $F$. For each place $v$ of $F$, we define
\begin{equation*} \label{isogeny_local_conds}
\msW_{\varphi, v} = \ker\left(H^1(G_v, A[\varphi]) \to H^1(G_v, A)\right).
\end{equation*}
Taking $\msW_{\varphi}$ to be the product of these groups, the $\varphi$-Selmer group is then defined by
\[\Sel^{\varphi} A = \Sel\left(A[\varphi],\, \msW_{\varphi}\right).\]
The next proposition shows these objects behave well under duality.
\begin{prop}
\label{prop:dual_isog_loc}
Taking $\varphi^{\vee} \colon B^{\vee} \to A^{\vee}$ to be the dual isogeny to $\varphi$, the   isomorphism
\[B^{\vee}\left[\varphi^{\vee}\right]\, \cong \,A[\varphi]^\vee\]
provided by the Weil pairing extends to an isomorphism
\[\left(B^{\vee}\left[\varphi^{\vee}\right], \, \msW_{\varphi^{\vee}}\right) \,\cong\, (A[\varphi], \msW_{\varphi})^{\vee}\]
in $\SMod_F$. 
\end{prop}
This appears as \cite[Proposition B.1]{Cesn17}. There, it is reduced to the case that $\varphi$ is multiplication by a positive integer, which is implicit in  work of Milne \cite[Section I.3]{Milne86}. 

In addition to the isogeny $\varphi \colon A \to B$, we now choose another isogeny $\lambda \colon C \to A$ of abelian varieties over $F$ whose degree is indivisible by the characteristic of $F$. We then have a short exact sequence
\[E \,=\,\Big[0 \to \big(C[\lambda], \,\msW_{\lambda}\big) \to \big(C[\varphi \circ \lambda], \,\msW_{\varphi\, \circ\, \lambda}\big) \xrightarrow{\,\, \lambda\,\,} \big(A[\varphi],\, \msW_{\varphi}\big) \to 0\Big]\]
in $\SMod_F$. The Cassels--Tate pairing for $E$ takes the form
\[\CTP_E \colon \Sel^{\varphi} A \,\times\, \Sel^{\lambda^{\vee}} A^{\vee} \to \QQ/\Z,\]
and the pairing's kernels are 
\[\lambda\left(\Sel^{\varphi \,\circ\, \lambda} C\right)\quad\text{and}\quad \varphi^{\vee}\big(\Sel^{\lambda^{\vee} \circ\, \varphi^{\vee}} B^{\vee}\big).\]
This pairing is just the one induced by the specialization of the classical pairing \eqref{eq:OGCT} to $\Sha(A)[\varphi]$ and $\Sha(A^{\vee})[\lambda^{\vee}]$. For $A/K$ an elliptic curve over a number field, the kernels of this pairing were first calculated in \cite[Theorem 3]{Fish03}.

An advantage of using this incarnation of the classical Cassels--Tate pairing is that the exact sequence $E$ may be simpler than any of the sequences
\[0 \to \big(A[n], \,\msW_n\big) \to \big(A[n^2], \,\msW_{n^2}\big)\to \big(A[n], \,\msW_n\big) \to 0\]
we might otherwise use to calculate the Cassels--Tate pairing. This is particularly clear in the case that $E$ splits as an exact sequence of $G_F$-modules, where the Cassels--Tate pairing takes a particularly simple form (see Definition \ref{defn:local_pairings}). This observation plays a major role in work of McCallum \cite{MR952286}, who considered the Cassels--Tate for isogeny Selmer groups as above in the case where $\lambda$ and $\varphi$ are both endomorphisms of $A$.

\subsection{Split exact sequences and local pairings} \label{sec:split_local_pairings}

Suppose we have an exact sequence 
\begin{equation} \label{eq:split_sequence_E}
E\, =\, \left[0 \to (M_1,\, \msW_1) \xrightarrow{\,\iota \,} (M, \,\msW) \xrightarrow{\pi} (M_2, \,\msW_2) \to 0\right]
\end{equation}
in $\textup{SMod}_{F}$. If this sequence is isomorphic to the split exact sequence
\[0 \to (M_1,\, \msW_1) \xrightarrow{\,\quad\,} (M_1\oplus M_2, \,\msW_1\oplus \msW_2) \xrightarrow{\quad} (M_2, \,\msW_2) \to 0\]
 then the resulting Cassels--Tate pairing is trivial. However, it is possible that the underlying sequence of $G_{F}$-modules is split, but the local conditions are such that the resulting pairing is non-trivial. In this situation, the Cassels--Tate pairing factors through a purely local pairing, as we now explain. An instance of this phenomenon for the Cassels--Tate pairing associated to an abelian variety appears in work of McCallum \cite{MR952286}.
 
 \begin{defn}
\label{defn:local_pairings}
 Given an exact sequence $E$ in $\textup{SMod}_{F}$ as in \eqref{eq:split_sequence_E}, suppose we have a $G_{F}$-equivariant homomorphism
 \begin{equation} \label{eq:section_to_pi}
 s \colon M_2 \to M
 \end{equation}
 giving a section to $\pi$ in the category of $G_F$-modules. We can then define a local pairing 
 \[\textup{LP}_{E,s} \colon \msW_2 \otimes \msW_1^\perp \to \mathbb{Q}/\mathbb{Z} \]
 as follows. Given tuples $\phi=(\phi_v)_v$ in $\msW_2$ and $\psi=(\psi_v)_{v}$ in $\msW_1^\perp$, choose $(\phi_{v, M})_v$ in $\msW$ so $\phi_v = \pi(\phi_{v, M})$ for all places $v$ of $F$. We then define
\[\textup{LP}_{E,s}\big(\phi, \,\psi\big) = \sum_v \inv_v\left(\iota^{-1}\left( s \circ \phi_v - \phi_{v, M}\right) \cup_{M_1} \psi_v\right).\]
Given instead global elements $\phi$ in $\Sel\, M_2$ and $\psi$ in $\Sel\, M_1^{\vee}$, we then have
\begin{equation}
\label{eq:local_pairings}
\CTP_E\left(\phi,\, \psi\right) = \textup{LP}_{E, s}\left((\res_{G_v} \phi)_v,\, (\res_{G_v} \psi)_v\right),
\end{equation}
as can be verified directly from Definition \ref{defn:CTP}.
\end{defn}

 By non-degeneracy of the local Tate pairings \eqref{eq:local_tate_pairing}, one sees that the left kernel of $\textup{LP}_{E,s}$ is equal to $\msW_2\cap s^{-1}(\msW)$, so that this pairing measures the failure of $s$ to map $\msW_2$ into $\msW$. In this simple case, the Cassels--Tate pairing only encodes these local lifting questions.

\subsection{Sums of local conditions}
The following special case of $\text{LP}$ plays an essential role in Mazur--Rubin's proof of \cite[Proposition 1.3]{MR2373150}. A certain case of this construction also appears in work of Howard \cite[Lemma 1.5.6]{MR2098397}.

Suppose we have objects $(M,\msW_a)$ and $(M,\msW_b)$ in $\SMod_F$. We have an exact sequence
 \begin{equation} \label{eq:sum_condition_exact_seq}
E\, =\, \left[0 \to (M,\, \msW_a\cap \msW_b) \xrightarrow{\,\Delta \,} (M\oplus M, \,\msW_a\oplus \msW_b) \xrightarrow{~} (M, \,\msW_a+\msW_b) \to 0\right] 
\end{equation}
in $\SMod_F$. Here the first map sends $m$ to $(m,m)$ and the second map sends $(m_a,m_b)$ to $m_a-m_b$ for any $m$, $m_a$ and $m_b$ in $M$. The Cassels--Tate pairing then gives a non-degenerate pairing
\[\CTP_E \colon \Sel (M, \,\msW_a + \msW_b)  \,\times \,\Sel\left(M^\vee,\,\msW_a^\perp+\msW_b^\perp\right)\longrightarrow \QQ/\Z\]
with left and right kernels
 \[\Sel(M,\msW_a)+\Sel(M,\msW_b)\quad  \textup{ and }  \quad \Sel(M^\vee,\msW_a^\perp)+\Sel(M^\vee,\msW_b^\perp),\]
respectively.

As a sequence of $G_F$-modules, \eqref{eq:sum_condition_exact_seq}  is a split extension, with the map $s \colon M \to M \oplus M$ defined by $s(m) = (m, 0)$ giving a splitting. So $\CTP_E$ is equal to the pairing $\text{LP}_{E, s}$ defined above. Concretely, suppose we have $\phi \in \Sel\, M$ and $\psi \in \Sel\, M^{\vee}$, and choose decompositions
\begin{alignat*}{3}
&(\res_{G_v}\phi)_v &&= (\phi_{va} + \phi_{vb})_v &&\quad\text{with}\quad (\phi_{va})_v\in \msW_a,\,\,\, \,\, (\phi_{vb})_v\in \msW_b\quad\text{and}\\
&(\res_{G_v}\psi)_v &&= (\psi_{va} + \psi_{vb})_v &&\quad\text{with}\quad (\psi_{va})_v\in \msW_a^{\perp},\,\,  (\psi_{vb})_v\in \msW_b^{\perp}.
\end{alignat*}
Then, per \eqref{eq:local_pairings}, the Cassels--Tate pairing for \eqref{eq:sum_condition_exact_seq} takes the explicit form
\begin{equation}
\label{eq:loc_conds_sum}
\CTP_E(\phi, \psi) = \sum_v \inv_v\left( \phi_{vb} \cup_M (\psi_{va} + \psi_{vb})\right) = \sum_v \inv_v\left( \phi_{vb} \cup_M \psi_{va}\right),
\end{equation}
with the last equality following since $(\phi_{vb})_v$ and $(\psi_{vb})_v$ are orthogonal.

\subsubsection{Alternating pairings from sums of local conditions}
Continuing the above example, suppose there is a $G_F$-equivariant isomorphism $f \colon M \isoarrow M^{\vee}$ that satisfies $f(\msW_a) = \msW_a^{\perp}$ and $f(\msW_b) = \msW_b^{\perp}$, so we have isomorphisms
\[f \colon (M, \msW_a) \isoarrow (M, \msW_a)^{\vee} \quad\text{and}\quad  (M, \msW_b) \isoarrow (M, \msW_b)^{\vee}\]
in $\SMod_F$. We then have a diagram
\[\begin{tikzcd}[column sep = small]
0 \ar{r} & (M, \,\msW_a \cap \msW_b) \ar{r}\ar{d}{f} & (M \oplus M, \,\msW_a \oplus \msW_b) \ar{r}\ar{d}{(f, -f)} & (M, \msW_a + \msW_b) \ar{d}{-f} \ar{r} & 0\\
0 \ar{r} & (M, \, \msW_a + \msW_b)^{\vee} \ar{r}& \left(M^{\vee} \oplus M^{\vee}, \,\msW_a^{\perp} \oplus \msW_b^{\perp}\right)   \ar{r}& (M, \, \msW_a\cap \msW_b)^{\vee}\ar{r} & 0
\end{tikzcd}\]
giving an isomorphism of the short exact sequences $E$ and $E^{\vee}$ in $\SMod_F$. Combining naturality with the duality identity, we find that the resulting perfect pairing
\[\CTP_E(\text{--},\, f(\text{--}))\colon \frac{\Sel(M,\, \msW_a + \msW_b)}{\Sel(M, \, \msW_a) + \Sel(M, \, \msW_b)}\,\times\, \frac{\Sel(M,\, \msW_a + \msW_b)}{\Sel(M, \, \msW_a) + \Sel(M, \, \msW_b)} \xrightarrow{\quad} \QQ/\Z\]
 is antisymmetric.

Suppose now that  $f$ is associated  to some theta group $\mcH$ sitting in the sequence 
\[0 \to (\Fsep)^{\times}\xrightarrow{\,\,\iota_{\mcH}\,\,} \mcH \xrightarrow{\,\,\pi_{\mcH}\,\,} M \to 0,\]
and suppose that $\msW_a$ and $\msW_b$ are isotropic with respect to this theta group. We can construct a theta group over $M \oplus M$ as the quotient
\[\mcH_0 = \frac{\mcH \times \mcH}{\left\{\big(\iota_{\mcH}(\alpha),\, \iota_{\mcH}(\alpha)\big)\,\colon\,\, \alpha \in (\Fsep)^{\times}\right\}},\]
with the map from $(\Fsep)^{\times} \to \mcH_0$ corresponding to the assignment $\alpha \mapsto (\iota(\alpha), 0)$, and with the map $\mcH_0 \to M \oplus M$ corresponding to $(\pi_{\mcH}, \pi_{\mcH})$.

The local conditions $\msW_a \oplus \msW_b$ are isotropic with respect to this theta group by functoriality. Furthermore, given a set-section $s_{\mcH}\colon M \to \mcH$ of $\pi_{\mcH}$,   the assignment
\[m \xmapsto{\,\,\,\,} \big(s_{\mcH}(m), \,s_{\mcH}(m)\big)\]
defines an equivariant homomorphism from $M$ to $\mcH_0$ that lifts the diagonal map from $M$ to $M \oplus M$. The Poonen--Stoll class associated to $E$ and $\mcH_0$ is thus trivial, and by Theorem \ref{thm:theta_main}  the pairing $\CTP_E(\text{--},\, f(\text{--}))$ is an alternating perfect pairing on the quotient
\begin{equation}
\label{eq:square_sum_quotient}
\Sel(M,\, \msW_a + \msW_b)\big/\big(\Sel(M, \, \msW_a) + \Sel(M, \, \msW_b)\big).
\end{equation}
In this case, the quotient always has square order.

\begin{rmk}
A variant of the above construction appears in Klagsbrun--Mazur--Rubin \cite[Theorem 3.9]{KMR13}. Taking $M$ to be a vector space over $\FFF_2$, they show that the quotient \eqref{eq:square_sum_quotient} has even dimension given the existence of suitable Tate quadratic forms as defined in \cite[Definition 3.2]{KMR13}. In our situation, the relevant Tate quadratic forms are given by the connecting maps $q_{\mcH, G_v}$ associated to the theta group $\mcH$ as in Notation \ref{notat:theta_connecting}.

If $M$ is instead a vector space over $\FFF_p$ with $p$ odd, the pairing $\CTP_E(\text{--},\, f(\text{--}))$ is automatically alternating, and \eqref{eq:square_sum_quotient} automatically has even dimension. This is the key consequence used in the work of Howard \cite{MR2098397} and Mazur--Rubin \cite{MR2373150}.
\end{rmk}

\begin{rmk}
The fact that $\CTP_E(\text{--},\, f(\text{--}))$ is antisymmetric can also be proved as a simple consequence of \eqref{eq:loc_conds_sum} and the graded commutativity of cup product. Similarly, the fact that this pairing is alternating if $f$ is the associated map to $\mcH$ follows from \eqref{eq:loc_conds_sum}, Zarhin's identity \eqref{eq:Zark}, and the identity
\[\sum_v \inv_v(q_{\mcH, G_v}(\phi)) = 0,\]
which holds for all $\phi$ in $H^1(G_F, M)$. 
\end{rmk}

\bibliography{references}{}

\providecommand{\bysame}{\leavevmode\hbox to3em{\hrulefill}\thinspace}
\providecommand{\MR}{\relax\ifhmode\unskip\space\fi MR }
\providecommand{\MRhref}[2]{%
  \href{http://www.ams.org/mathscinet-getitem?mr=#1}{#2}
}
\providecommand{\href}[2]{#2}
\begin{thebibliography}{10}

\bibitem{BlKa07}
Spencer Bloch and Kazuya Kato, \emph{L-functions and {T}amagawa numbers of
  motives}, The Grothendieck Festschrift, Springer, 2007, pp.~333--400.

\bibitem{Buhl10}
Theo B\"{u}hler, \emph{Exact categories}, Expo. Math. \textbf{28} (2010),
  no.~1, 1--69.

\bibitem{Cass62}
John~W.S. Cassels, \emph{Arithmetic on curves of genus 1. {IV}. {P}roof of the
  {H}auptvermutung.}, Journal f{\"u}r die reine und angewandte Mathematik
  \textbf{1962} (1962), no.~211, 95--112.

\bibitem{Cesn17}
K{\c e}stutis {\v C}esnavi{\v c}ius, \emph{{$p$}-{S}elmer growth in extensions
  of degree {$p$}}, J. Lond. Math. Soc. (2) \textbf{95} (2017), no.~3,
  833--852.

\bibitem{Fish03}
Tom~A Fisher, \emph{The {C}assels--{T}ate pairing and the {P}latonic solids},
  Journal of Number Theory \textbf{98} (2003), no.~1, 105--155.

\bibitem{Flach90}
Matthias Flach, \emph{A generalisation of the {C}assels-{T}ate pairing},
  Journal f{\"u}r die reine und angewandte Mathematik \textbf{412} (1990),
  113--127.

\bibitem{MR2329311}
Norbert Hoffmann and Markus Spitzweck, \emph{Homological algebra with locally
  compact abelian groups}, Adv. Math. \textbf{212} (2007), no.~2, 504--524.
  \MR{2329311}

\bibitem{MR2098397}
Benjamin Howard, \emph{The {H}eegner point {K}olyvagin system}, Compos. Math.
  \textbf{140} (2004), no.~6, 1439--1472.

\bibitem{KMR13}
Zev Klagsbrun, Barry Mazur, and Karl Rubin, \emph{Disparity in {S}elmer ranks
  of quadratic twists of elliptic curves}, Annals of Mathematics (2013),
  287--320.

\bibitem{MR597871}
Kenneth Kramer, \emph{Arithmetic of elliptic curves upon quadratic extension},
  Trans. Amer. Math. Soc. \textbf{264} (1981), no.~1, 121--135.

\bibitem{MR2373150}
Barry Mazur and Karl Rubin, \emph{Finding large {S}elmer rank via an arithmetic
  theory of local constants}, Ann. of Math. (2) \textbf{166} (2007), no.~2,
  579--612.

\bibitem{MR952286}
William~G. McCallum, \emph{On the {S}hafarevich-{T}ate group of the {J}acobian
  of a quotient of the {F}ermat curve}, Invent. Math. \textbf{93} (1988),
  no.~3, 637--666.

\bibitem{Milne86}
James~S. Milne, \emph{Arithmetic duality theorems}, Perspectives in
  Mathematics, vol.~1, Academic Press, Inc., Boston, MA, 1986.

\bibitem{MR3951582}
Adam Morgan, \emph{Quadratic twists of abelian varieties and disparity in
  {S}elmer ranks}, Algebra Number Theory \textbf{13} (2019), no.~4, 839--899.

\bibitem{MoSm21b}
Adam Morgan and Alexander Smith, \emph{The {C}assels--{T}ate pairing for finite
  commutative group schemes}, In preparation.

\bibitem{MoSm21c}
\bysame, \emph{Field change for the {C}assels-{T}ate pairing and applications
  to class groups}, arXiv preprint arXiv:2207.05674 (2022).

\bibitem{Morris77}
Sidney~A Morris, \emph{Pontryagin duality and the structure of locally compact
  abelian groups}, vol.~29, Cambridge University Press, 1977.

\bibitem{Mumf66}
David Mumford, \emph{On the equations defining abelian varieties. {I}},
  Inventiones mathematicae \textbf{1} (1966), no.~4, 287--354.

\bibitem{Neko06}
Jan Nekov\'{a}\v{r}, \emph{Selmer complexes}, Ast\'{e}risque (2006), no.~310,
  viii+559.

\bibitem{Neuk08}
J\"{u}rgen Neukirch, Alexander Schmidt, and Kay Wingberg, \emph{Cohomology of
  number fields}, second ed., Grundlehren der Mathematischen Wissenschaften,
  vol. 323, Springer-Verlag, Berlin, 2008.

\bibitem{PR11}
Bjorn Poonen and Eric Rains, \emph{Self cup products and the theta
  characteristic torsor}, Mathematical Research Letters \textbf{18} (2011),
  no.~6, 1305--1318.

\bibitem{PR12}
\bysame, \emph{Random maximal isotropic subspaces and {S}elmer groups}, Journal
  of the American Mathematical Society \textbf{25} (2012), no.~1, 245--269.

\bibitem{Poon99}
Bjorn Poonen and Michael Stoll, \emph{The {C}assels-{T}ate pairing on polarized
  abelian varieties}, Annals of Mathematics \textbf{150} (1999), no.~3,
  1109--1149.

\bibitem{Rich77}
Fred Richman and Elbert Walker, \emph{Ext in pre-abelian categories}, Pacific
  Journal of Mathematics \textbf{71} (1977), no.~2, 521--535.

\bibitem{Schn99}
Jean-Pierre Schneiders, \emph{Quasi-abelian categories and sheaves}, M\'{e}m.
  Soc. Math. Fr. (N.S.) (1999), no.~76.

\bibitem{MR1466966}
Jean-Pierre Serre, \emph{Galois cohomology}, Springer-Verlag, Berlin, 1997,
  Translated from the French by Patrick Ion and revised by the author.

\bibitem{Smi22}
Alexander Smith, \emph{The distribution of $\ell^{\infty}$-{S}elmer groups in
  degree $\ell$ twist families {I}}, arXiv preprint arXiv:2207.05674 (2022).

\bibitem{Tate63}
John Tate, \emph{Duality theorems in {G}alois cohomology over number fields},
  Proc. {I}nternat. {C}ongr. {M}athematicians ({S}tockholm, 1962), Inst.
  Mittag-Leffler, Djursholm, 1963, pp.~288--295.

\bibitem{Tate76}
\bysame, \emph{Relations between {$K_{2}$} and {G}alois cohomology}, Invent.
  Math. \textbf{36} (1976), 257--274.

\bibitem{Zark74}
Yu~G. Zarkhin, \emph{Noncommutative cohomologies and {M}umford groups},
  Mathematical notes of the Academy of Sciences of the USSR \textbf{15} (1974),
  no.~3, 241--244.

\end{thebibliography}
\bibliographystyle{amsplain}

\end{document}